\documentclass[11pt, oneside,american]{article}
\usepackage[a4paper]{geometry}

\usepackage[utf8]{inputenc}

\usepackage{amssymb,amsmath,bbm,xfrac,array}

\usepackage{amsthm}
\begingroup
    \makeatletter
    \@for\theoremstyle:=definition,remark,plain\do{%
        \expandafter\g@addto@macro\csname th@\theoremstyle\endcsname{%
            \addtolength\thm@preskip\parskip
            }%
        }
\endgroup

\usepackage{hyperref}
\usepackage[style=alphabetic,backend=biber,bibencoding=utf8,url=false,isbn=false,doi=false]{biblatex}

\usepackage{parskip}

\usepackage[svgnames]{xcolor}

\newtheorem{theorem}{Theorem}
\newtheorem{proposition}[theorem]{Proposition}
\newtheorem{lemma}[theorem]{Lemma}
\newtheorem{assumption}[theorem]{Assumption}
\newtheorem{corollary}[theorem]{Corollary}
\theoremstyle{definition}

\theoremstyle{remark}

\newtheoremstyle{todo}%
  {}{}%
  {\color{blue}}{}%
  {\color{blue}\bfseries}{}%
  {.5em}{}

\theoremstyle{todo}

\newcommand{\gaptext}[1]{#1}

\newtheoremstyle{problem}%
  {}{}%
  {}{}%
  {\color{red}\bfseries}{}%
  {.5em}{}
\theoremstyle{problem}

\DeclareMathOperator{\Var}{Var}
 \DeclareMathOperator{\Cov}{Cov}
\DeclareMathOperator{\Ent}{Ent}

\DeclareMathOperator{\Tr}{Tr}

\DeclareMathOperator{\HH}{H}
\DeclareMathOperator{\I}{I}
\DeclareMathOperator{\osc}{osc}

\newcommand{\E}{\mathbb E}

\newcommand{\PP}{\mathbb P}
\newcommand{\QQ}{\mathbb Q}

\newcommand{\abs}[1]{\lvert#1\rvert}
\newcommand{\Abs}[1]{\Bigl\lvert#1\Bigr\rvert}
\newcommand{\QV}[1]{\langle#1\rangle}

\newcommand{\lipnorm}[1]{\lVert#1\rVert_\mathrm{Lip}}

\newcommand{\norm}[1]{\lVert#1\rVert}

\newcommand{\R}{\mathbb{R}}
\newcommand\leqm{\mathrel{\overset{\mathrm{(m)}}{\leq}}}
\newcommand{\Id}{\text{Id}}

\title{Towards a Quantitative Averaging Principle for Stochastic Differential Equations}
\author{Bob Pepin \thanks{bob.pepin@uni.lu}}

\begin{document}

\maketitle

\begin{abstract}
  This work explores the use of a forward-backward martingale method
  together with a decoupling argument and entropic estimates between
  the conditional and averaged measures to prove a strong averaging
  principle for stochastic differential equations with order of
  convergence 1/2. We obtain explicit expressions for all the
  constants involved. At the price of some extra assumptions on the time
  marginals and an exponential bound in time, we loosen the usual
  boundedness and Lipschitz assumptions. We conclude with an
  application of our result to Temperature-Accelerated Molecular
  Dynamics.
\end{abstract}

\section{Introduction and notation}
\subsection{Motivation and main result}

We are interested in stochastic differential equations of the form
\begin{align*}
  dX_t &= {\varepsilon}^{-1} b_X(X_t, Y_t) dt + {\varepsilon}^{-1/2} \sigma_X(X_t, Y_t)
         dB^X_t, \quad X_0 = x_0, \\
  dY_t &= b_Y(X_t, Y_t)dt + \sigma_Y(Y_t) dB^Y_t, \quad Y_0 = y_0
\end{align*}
for some $\varepsilon > 0$ and $x_0 \in \R^n$, $y_0 \in \R^m$. The
precise assumptions on the coefficients are stated in
Assumption~\ref{ass:coeffs} and they essentially amount to $b_X$ being
one-sided Lipschitz outside a compact set, $b_Y$
being differentiable with bouded derivative, $\sigma_X$ being bounded
and the process
being elliptic.

It is well known (see for example \autocite{freidlin_random_2012})
that when all the coefficients and their first derivatives are
bounded, $Y$ (which depends on $\varepsilon$) can be approximated by a
process $\bar{Y}$ on $\R^m$ in the sense that for all $T > 0$ fixed
\begin{equation*}
  \PP\left(\sup_{0\leq t \leq T} \abs{Y_t - \bar{Y}_t} > \varepsilon\right) \to 0
  \text{ as } \varepsilon \to 0.
\end{equation*}
The process $\bar{Y}$ solves the SDE
\begin{equation*}
  d\bar{Y}_t = \bar{b}_Y(\bar{Y}_t) dt + \sigma_Y(\bar{Y}_t) dB^Y_t,
  \quad \bar{Y}_0 = y_0
\end{equation*}
with
\begin{equation*}
  \bar{b}(y) = \int_{\R^n} b_Y(x, y) \mu^y(dx).
\end{equation*}
Here ${(\mu^y)}_{y \in \R^m}$ is a family of measures on $\R^n$ such
that for each $y$, $\mu^y$ is the unique stationary measure of $X^y$ with
\begin{equation*}
  dX^y_t = b_X(X^y_t, y) dt + \sigma_X(X^y_t, y) dB^X_t.
\end{equation*}

The work \autocite{liu_strong_2010,} replaces the boundedness
assumption on $b_X$ and $\sigma_X$ by a dissipativity condition and
shows the following rate of convergence of the time marginals:
\begin{equation*}
  \sup_{0\leq t \leq T} \E\abs{Y_t - \bar{Y}_t} \leq C\varepsilon^{1/2}
\end{equation*}
for some constant $C$ independent of $\varepsilon$.

In \autocite{legoll_pathwise_2016,} the author relax the growth
conditions on the coefficients of the SDE and show that when
$(X_t, Y_t)$ is a reversible diffusion process with stationary measure
$\mu = e^{-V(x, y)}dx dy$ such that for each $y$, a Poincaré
inequality holds for $e^{-V(x, y)}dx$, then there exists a constant $C$
independent of $\varepsilon$ such that
\begin{equation*}
  \E \sup_{0\leq t \leq T}\abs{Y_t - \bar{Y}_t} \leq C \varepsilon^{1/2}.
\end{equation*}

The present work extends the approach from
\autocite{legoll_pathwise_2016,} to the non-stationary case and drops
the boundedness assumption on $b_Y$, $\sigma_Y$ commonly found in the
averaging literature. The
general setting and notation will be outlined in
Section~\ref{sec:setting}. Section~\ref{sec:fbmart} presents a
forward-backward martingale argument under the assumption of a
Poincaré inequality for the regular conditional probability density
$\rho^y_t$ of $X_t$ given $Y_t = y$. By dropping the stationarity
assumption, we have to deal with the fact that $\rho^y_t$ is no longer
equal to $\mu^y$ defined above. This is done in Section~\ref{sec:dist}
by developing the relative entropy between $\rho^y_t$ and $\mu^y$
along the trajectories of $Y$. Dropping the boundedness assumption on
$b_Y$ forces us to consider the mutual interaction between $X_t$ and
$Y_t$. In Section~\ref{sec:decoupling} we address this problem when
the timescales of $X$ and $Y$ are sufficiently separated. The main
theorem is proven in Section~\ref{sec:mainthm}.
Section~\ref{sec:examples} applies the theorem to a particular class
of SDEs to
obtain sufficient conditions such
that for any $T > 0$ and $\varepsilon$ sufficiently small
\begin{equation*}
  \E \sup_{0\leq t\leq T} \Abs{Y_t - \bar{Y}_t} \leq C \varepsilon^{1/2}
\end{equation*}
where $C$ will be explicitly
given in terms of the coefficients of the SDE and the
Poincaré constant for $\rho^y_t$.
\subsection{Setting and notation}\label{sec:setting}
The results in sections 2 to 5 will be stated in the setting of an SDE on $\mathcal{X} \times
\mathcal{Y} = \R^n \times \R^m$ of the form
\begin{align*}
  dX_t &= b_X(X_t, Y_t)dt + \sigma_X(X_t, Y_t) dB^X_t,\,X_0 = x \\
  dY_t &= b_Y(X_t, Y_t)dt + \sigma_Y(X_t, Y_t) dB^Y_t,\,Y_0 = y
\end{align*}
where $x \in {\cal X} = \R^n$, $y \in {\cal Y} = \R^m$, $B^X, B^Y$ are
independent standard Brownian motions on $\R^n$ and
$\R^m$ respectively and $b_X = (b^i_X)_{1\leq i \leq n}$, $b_Y =
(b_Y^i)_{1\leq i\leq m}$, $\sigma_X$ and $\sigma_Y$ are
continuous mappings from ${\cal X}\times{\cal Y}$ to $\mathcal{X}$,
$\mathcal{Y}$, $\R^{n\times n}$ and $\R^{m\times m}$ respectively.

The matrices $A_X = {(a_X^{ij}(x, y))}_{i,j\leq n}$ and $A_Y =
{(a_Y^{ij}(x, y))}_{i,j\leq m}$ are defined by
\begin{align*}
  A_X(x, y) &= \tfrac12 \sigma_X(x, y)
            {\sigma_X(x, y)}^T, &
  A_Y(x, y) &= \tfrac12 \sigma_Y(x,
  y) {\sigma_Y(x, y)}^T
\end{align*}
and the infinitesimal generator $L$ of $(X, Y)$ has a decomposition $L
= L^X + L^Y$ such that
\begin{align*}
  L^X f &= \sum_{i=1}^n b_X^i \partial_{x_i} f +
                \sum_{i,j=1}^n a_X^{ij} \partial_{x_i x_j}^2
                f, \\
  L^Y f &= \sum_{i=1}^m b_Y^i \partial_{x_i} f +
                \sum_{i,j=1}^m a_Y^{ij} \partial_{x_i x_j}^2
                f, \\
  L f  &= (L^X + L^Y)f.
\end{align*}

We will also make use of the square field operators $\Gamma$ and
$\Gamma^X$, defined by
\begin{align*}
  \Gamma(f, g) &= \tfrac12(L(fg) - gLf - fLg) =
  \sum_{i,j=1}^n a_X^{ij} \partial_{x_i} f \partial_{x_j}
  g +  \sum_{i,j=1}^m a_Y^{ij} \partial_{y_i} f \partial_{y_j}
  g, \\
  \Gamma^X(f, g) &= \tfrac12(L^X(fg) - gL^X f - fL^X g) =
  \sum_{i,j=1}^n a_X^{ij} \partial_{x_i} f \partial_{x_j}
  g.
\end{align*}

We denote $\rho_t(dx, dy)$ the marginal distribution of $(X, Y)$ at time
$t$, i.e.\ for $\varphi \in C_c^\infty$
\begin{equation*}
  \E[\varphi(X_t, Y_t)] = \int_{\mathcal{X}\times\mathcal{Y}} \varphi(x, y) \rho_t(dx, dy)
\end{equation*}
and we let $\rho_t^y(dx)$ be the regular conditional probability density
of $P(X_t \in dx | Y_t = y)$.

If a measure $\mu(dx, dy)$ is absolutely continuous with respect to
Lebesgue measure we will make a slight abuse
of notation and denote $\mu(x, y)$ its density. 

We will also make use of a family of auxiliary processes ${(X^y)}_{y \in \mathcal{Y}}$ defined by
\begin{equation*}
  dX^y_t = b_X(X_t, y) dt + \sigma_X(X_t, y) dB^X_t,\, X^y_0 = x
\end{equation*}
which we assume to be uniformly ergodic and we denote $\mu^y$ the unique
stationary invariant measure of $X^y$.

We will furthermore use another auxiliary process $\tilde{X}$ solution to
\begin{equation*}
  d\tilde{X}_t = b_X(\tilde{X}_t, Y_t)dt + \sigma_X(\tilde{X}_t, Y_t) d\tilde{B}^X_t,\,\tilde{X}_0 = x \\
\end{equation*}
where $\tilde{B}^X$ is an $n$-dimensional Brownian motion independent
of $B^X$ and $B^Y$ and we denote $\tilde{\rho}^y_t$ the regular conditional
probability density of $P(\tilde{X}_t \in dx | Y_t = y)$.

For the section on decoupling and the main theorem we need in addition
to a separation of timescales the following regularity conditions on
the coefficients of $(X, Y)$:

\begin{assumption}\label{ass:coeffs} Regularity of the coefficients: 
  \begin{itemize}    
  \item $b_X$ verifies a one-sided Lipschitz condition with constant
    $\kappa_X$ and perturbation $\alpha$:
    \begin{equation*}
      {(x_1-x_2)}^T (b_X(x_1, y) - b_X(x_2, y)) \leq -\kappa_X
      \abs{x_1-x_2}^2 + \alpha \text{ for all } x_1, x_2 \in
      \mathcal{X}, y \in \mathcal{Y}
    \end{equation*}
  \item $b_Y$ has a bounded first derivative in $x$:
    \begin{equation*}
      {\kappa_Y}^2 := \frac{1}{m} \sum_{i=1}^m \sup_{x,y} \abs{\nabla_x b^i_Y(x, y)}^2 < \infty
    \end{equation*}
  \item $A_X$ is nondegenerate uniformly with respect to $(x,y)$,
    i.e.~there exist two constants
    $0 < \lambda_X \leq \Lambda_X < \infty$ such that the following
    matrix inequalities hold (in the sense of nonnegative
    definiteness):
    \begin{equation*}
      {\lambda_X} \Id \leq A_X(x, y) \leq {\Lambda_X} \Id
    \end{equation*}
  \item $\sigma_Y$ is invertible and $A_Y$ is uniformly elliptic with
    respect to $(x, y)$, i.e.~there exists a constant $\lambda_Y > 0$
    such that the following matrix inequality holds (in the sense of
    nonnegative definiteness):
    \begin{equation*}
      {\lambda_Y} \Id \leq A_Y(x, y)
    \end{equation*}
  \end{itemize}
\end{assumption}

\begin{assumption}\label{ass:rhoregularity}
Regularity of the time marginals:
  \begin{itemize}
  \item 
  There exists $M_0$ such that for $\abs{x}^2 + \abs{y}^2 > M_0$,
  $r > 0$, $\alpha > 0$
  \begin{equation*}
    {\nabla_x\log\rho_t(x,y)}^T x + {\nabla_y\log\rho_t(x,y)}^T y \leq -r
    (\abs{x}^2+\abs{y}^2)^{\alpha / 2}.
  \end{equation*}
\item
  The regular conditional probability densities $\tilde{\rho}^y_t$ of
  $P(\tilde{X}_t \in dx | Y_t = y)$ satisfy Poincaré inequalities with
  constants $c_P(y)$ independent of $\varepsilon$:
  \begin{equation*}
    \int(f - \tilde{\rho}^y_t(f))^2 d\tilde{\rho}^y_t \leq c_P(y) \int
    \abs{\sigma_X \nabla_x f}^2 d\tilde{\rho}^y_t.
  \end{equation*}
\end{itemize}

\end{assumption}

In order to characterise the separation of timescales, we introduce a
parameter $\gamma$ defined by
\begin{equation*}
  \gamma =
  \frac{{\kappa_X}^2{\lambda_Y}}{{\Lambda_X}{\kappa_Y}^2}.
\end{equation*}

\section{Approximation by conditional expectations}\label{sec:fbmart}
We will start with a Lemma for a form of the Lyons-Meyer-Zheng forward-backward martingale
decomposition.

\begin{lemma}[Forward-backward martingale decomposition]\label{lemma:fbmart}
  For a diffusion process $\xi_t$ with generator $L_t$ and square field
  operator $\Gamma_t$ we have for $f(s, \cdot) \in \mathcal{D}(L_s +
  \tilde{L}_{T-s})$ and $1\leq p\leq2$
  \begin{equation*}
    \E \sup_{0\leq t\leq T} \Abs{\int_0^t -(L_s+\tilde{L}_{T-s}) f(s, \xi_s)
      ds}^p \leq 3^{p-1} (2C_p + 1) {\left(\E  \int_0^T 2 \Gamma_t(f) (\xi_t) dt \right)}^{p/2}
    \end{equation*}
    where $\tilde{L}_s$ is the generator of the time-reversed process
    $\tilde{\xi}_t = \xi_{T-t}$ and $C_p$ is the constant in the upper
    bound of the Burkholder-Davis-Gundy inequality for $L^p$.
\end{lemma}

\begin{proof} 

First, suppose that $f(t, x)$ is once differentiable in
$t$ and twice differentiable in $x$ so that we can apply the Itô
formula. 

We express $f(t, \xi_t) - f(0, \xi_0)$ in two different
ways, using the fact that $\xi_t = \tilde{\xi}_{T-t}$:
\begin{align*}
  f(t, \xi_t) - f(0, \xi_0) 
  &= \int_0^t (\partial_s + L_s) f(s, \xi_s) ds + M_t \tag{1} \\
  f(0, \xi_0) - f(t, \xi_t) 
  &= (f(0, \tilde{\xi}_T) - f(T, \tilde{\xi}_0)) - (f(t, \tilde{\xi}_{T-t}) - f(T, \tilde{\xi}_0)) \\
  & = \int_{T-t}^T (-\partial_s + \tilde{L}_s) f(T-s, \tilde{\xi}_s)
    ds + \tilde{M}_T - \tilde{M}_{T-t} \\
  &= \int_0^t (-\partial_s + \tilde{L}_{T-s}) f(s, \tilde{\xi}_{T-s}) ds + \tilde{M}_T - \tilde{M}_{T-t} \\
  &= \int_0^t (-\partial_s + \tilde{L}_{T-s}) f(s, \xi_s) ds + \tilde{M}_T - \tilde{M}_{T-t} \tag{2}
\end{align*}
where $M$ and $\tilde{M}$ are martingales with
\begin{align*}
  \QV{M}_T &= \int_0^T 2\Gamma_s(f) (s, \xi_s) ds, \\
  \QV{\tilde{M}}_T &= \int_0^T 2 \Gamma_{T-s}(f) (T-s, \tilde{\xi}_s) ds =
                     \int_0^T 2\Gamma_s(f) (s, \xi_s) ds =  \QV{M}_T.
\end{align*}
Summing (1) and (2), we get
\begin{align*}
  \int_0^t -(L_s + \tilde{L}_{T-s}) f(s, \xi_s) = M_t + \tilde{M}_T - \tilde{M}_{T-t}.
\end{align*}
We have by the Burkholder-Davis-Gundy $L^p$-inequality that
\begin{align*}
  \E \sup_{0 \leq t\leq T} \abs{M_t}^p &\leq C_p \E[\QV{M}_T^{p/2}] \\
  \E \sup_{0 \leq t\leq T} \abs{\tilde{M}_{T-t}}^p = \E \sup_{0 \leq t\leq T} \abs{\tilde{M}_t}^p 
                                       &\leq C_p \E[\QV{\tilde{M}}_T^{p/2}] = C_p \E[\QV{M}_T^{p/2}]
\end{align*}
so that
\begin{align*}
  \E\left[{\sup_{0\leq t\leq T}\Abs{\int_0^t -(L_s + \tilde{L}_{T-s}
  f)(s, \xi_s)}^p ds}\right]
  &= \E \sup_{0\leq t\leq T}\Abs{M_t + \tilde{M}_T - \tilde{M}_{T-t}}^p \\
  & \leq 3^{p-1} \left(\E \sup_{0\leq t\leq T}\abs{M_t}^p + \E \abs{\tilde{M}_T}^p +
    \E \sup_{0\leq t\leq T}\abs{\tilde{M}_{T-t}}^p \right) \\
  & \leq 3^{p-1} (2C_p + 1) {(\E {\QV{M}_T})}^{p/2} \\
   & \leq 3^{p-1} (2C_p + 1) {\left(\E  \int_0^T 2 \Gamma_t(f) (t, \xi_t) dt \right)}^{p/2}
\end{align*}

For a general $f(t, x)$, $C^2$ in $x$ and locally integrable in $t$,
we approximate first in space by stopping $\xi_t$ and then in time by
mollifying $f(\cdot, x)$.

For $R > 0, \varepsilon > 0$ and a function $f(t, x)$ we will use the
notation
\begin{align*}
{(f)}^R(t, x) &= f(t, x \frac{\abs{x} \wedge R}{\abs{x}}), \\
(f)_\varepsilon(t, x) &= \int_{-\infty}^{+\infty} f(s, x)
\phi_\varepsilon(t-s) ds 
\end{align*}
where $\phi_\varepsilon$ is a mollifier.
In particular, $(f)^R(t, \cdot)$ is bounded and
$(f)_\varepsilon(\cdot, x)$ is differentiable.

Let $K_t = L_t + \tilde{L}_{T-t}$. $K_t$ is a second order
partial differential operator and so can be written as
\begin{equation*}
  K_t f(t, x) = \sum b^i(t, x) \partial_{x_i} f(t,
  x) + \sum a^{ij}(t, x) \partial^2_{x_i x_j} f(t, x)
\end{equation*}
for some functions $b^i$ and $a^{ij}$.

Define the stopping times $\tau_R = \inf\{t > 0: \abs{\xi_t} \geq R\}$.
Then
\begin{align} \label{eq:21}
  \E \sup_{0\leq t\leq T} \Abs{\int_0^{t \wedge \tau_R} K_s (f)_\varepsilon(s, \xi_s) ds}^p 
  &= \E \sup_{0\leq t\leq \tau_R \wedge T} \Abs{\int_0^t (K_s (f)_\varepsilon)^R(s, \xi_s)
    ds}^p \notag \\ 
  &\leq \E \sup_{0\leq t\leq T} \Abs{\int_0^t (K_s
    (f)_\varepsilon)^R(s, \xi_s) ds}^p  \notag \\
  & \leq 3^{p-1} (2C_p + 1) {\left(\E  \int_0^T 2 (\Gamma_t ((f)_\varepsilon)^R (t, \xi_t) dt \right)}^{p/2}.
\end{align}

By differentiating inside the
integral for $(f)_\varepsilon$ we get
\begin{multline*}
  \int_0^{t\wedge \tau_R} K_s (f - (f)_\varepsilon)(s, \xi_s) ds 
  \leq \sup_{0
    \leq t \leq T, \abs{x} \leq R} \abs{b^i(t, x)} \int_0^T
  \sup_{\abs{x} \leq R} \abs{(\partial_{x_i} f - (\partial_{x_i}
    f)_\varepsilon)(s, x)} ds \\ + \sup_{0
    \leq t \leq T, \abs{x} \leq R} \abs{a^{ij}(t, x)} \int_0^T
  \sup_{\abs{x} \leq R} \abs{(\partial^2_{x_i x_j} f -
    (\partial^2_{x_i x_j}
    f)_\varepsilon)(s, x)} ds. 
\end{multline*} 

As $\varepsilon \to 0$, $(g)_\varepsilon \to g$ in
$L^1([0, T], L^\infty(B_R))$ and the integrals on the right hand side
go to $0$. We now let first $\varepsilon \to 0$
with dominated convergence and then $R \to \infty$ with monotone
convergence to get 
\begin{equation*}
  \E \sup_{0\leq t\leq T} \Abs{\int_0^{t} K_s f(s, \xi_s) ds}^p =
  \lim_{R \to \infty} \lim_{\varepsilon \to 0}
   \E \sup_{0\leq t\leq T} \Abs{\int_0^{t \wedge \tau_R} K_s (f)_\varepsilon(s, \xi_s) ds}^p 
\end{equation*}

For the right hand side of~\eqref{eq:21}, note that
\begin{equation*}
  (\Gamma_t (f) - \Gamma_t((f)_\varepsilon))^R = \Gamma_t(f-(f)_\varepsilon,
  f+(f)_\varepsilon)^R = (a^{ij})^R (\partial_{x_i} f - (\partial_{x_i}
  f)_\varepsilon)^R (\partial_{x_j} f + (\partial_{x_j}
  f)_\varepsilon)^R
\end{equation*}
so that
\begin{equation*}
  \int_0^T \abs{(\Gamma_t(f) - \Gamma_t((f)_\varepsilon))^R} \leq \sup_{0 \leq
    t \leq T, \abs{x}
    \leq R} a^{ij}(t, x) (\partial_{x_j} f + (\partial_{x_j}
  f)_\varepsilon) \int_0^T \sup_{\abs{x} \leq R} \abs{\partial_{x_i} f - (\partial_{x_i}
  f)_\varepsilon} dt.
\end{equation*}
Now the convergence follows again by first letting $\varepsilon \to 0$
with dominated convergence and then $R \to \infty$ with monotone
convergence.
\end{proof}

\begin{lemma}\label{lemma:gamma12}
  Let $L$ and $\hat{L}$ be generators of diffusion processes with common invariant
  measure $\mu$ and square field operators $\Gamma$ and $\hat{\Gamma}$
  respectively. Let $f,g$ be a pair of functions such that
  \begin{equation*}
    L f = \hat{L} g \text{ and } \int \hat{\Gamma}(f) d\mu \leq \int \Gamma(f) d\mu.
  \end{equation*}
  Then
  \begin{equation*}
    \int \Gamma(f) d\mu \leq \int\hat{\Gamma}(g) d\mu.
  \end{equation*}
\end{lemma}
\begin{proof}
  \begin{align*}
    \int \Gamma(f) d\mu 
    &= \int fLf d\mu = \int f \hat{L} g d\mu = \int \hat{\Gamma}(f, g) d\mu \\
    & \leq {\left(\int\hat{\Gamma}(f) d\mu\right)}^{1/2}{\left(\int\hat{\Gamma}(g) d\mu\right)}^{1/2} \\ 
    &\leq {\left(\int\Gamma(f) d\mu\right)}^{1/2}{\left(\int\hat{\Gamma}(g) d\mu\right)}^{1/2}.
  \end{align*}
  The result follows by dividing both sides by ${\left(\int\Gamma(f) d\mu\right)}^{1/2}$.
\end{proof}

\begin{lemma}\label{lemma:gammaL}
  Consider a generator $L$ with invariant measure $\mu$ and associated square field operator $\Gamma$.
  Assume that the following Poincaré inequality holds:
  \begin{equation*}
    \int {(\varphi - \mu(\varphi))}^2 d\mu \leq c_P \int \Gamma(\varphi) d\mu.
  \end{equation*}
  Then for any sufficiently nice $f$
  \begin{equation*}
    \int \Gamma(f) d\mu \leq c_P \int {(-Lf)}^2 d\mu \leq {c_P}^2 \int \Gamma(-Lf) d\mu
  \end{equation*}
\end{lemma}
\begin{proof}
  Since both $\Gamma$ and $L$ are differential operators, we can assume that $\mu(f) = 0$. Now,
  \begin{align*}
    {\left(\int \Gamma(f) d\mu\right)}^2 = {\left(-\int f L f d\mu\right)}^2 \leq \int f^2 d\mu \int {(-Lf)}^2 d\mu \leq c_P \int \Gamma(f) d\mu \int {(-Lf)}^2 d\mu
  \end{align*}
  and the first inequality follows after dividing both sides by $\int \Gamma(f) d\mu$. For the second inequality, we apply the Poincaré inequality again with $\varphi = (-Lf)$.
\end{proof}

\begin{proposition}\label{prop:sup}
  In the general setting of section \ref{sec:setting} with
  Assumption \ref{ass:rhoregularity} let $\nu_t^\eta(dx)$
  be the regular conditional probability density of $\PP(X_t \in dx |
  \phi(Y_t) = \eta)$ for a measurable function  $\phi \colon
  \mathcal{Y} \to \R^l$. If $\nu^\eta_t$ satisfies a Poincaré inequality with
  constant $c_P(\eta)$ independent of $t$ with respect to $\Gamma^X$ then
  for any function $f_t(x, y)$ with at most polynomial growth in $x$
  and $y$ such that $f_t(\cdot) \in
  C^2(\mathcal{X} \times \mathcal{Y})$, $\int_\mathcal{X} f_t(x,
  y) \nu_t^{\phi(y)}(dx) = 0$ and $1\leq p\leq2$
  \begin{multline*}
    \E \sup_{0\leq t\leq T}\Abs{\int_0^t f_s(X_s, Y_s)ds}^p \leq
    {3}^{p-1} 2^{-p/2} (2C_p + 1)
  {\left(\E \int_0^T {c_P}(\phi(Y_t))
   f_t^2(X_t, Y_t) dt \right)}^{p/2}
  \end{multline*}
  where $C_p$ is the constant in the upper bound of the
  Burkholder-Davis-Gundy inequality for $L^p$. 
\end{proposition}
\begin{proof}[Proof of Proposition~\ref{prop:sup}] The generator of
  the time-reversed process ${(X, Y)}_{T-t}$ is
  \autocite{haussmann_time_1986}
\begin{align*}
  \tilde{L}_t \varphi 
  &= -\sum_{i=1}^n b_X^i \partial_{x_i} \varphi - \sum_{i=1}^m b_Y^i
    \partial_{y_i} \varphi + \sum_{i,j=1}^n a^{ij}_X\partial^2_{x_i x_j} \varphi +
    \sum_{i,j=1}^m a^{ij}_Y\partial^2_{y_i y_j} \varphi \\
  & \quad + \frac{1}{p_{T-t}} \sum_{i,j=1}^n \partial_{x_j} (2 a^{ij}_X
    p_{T-t}) \partial_{x_i} \varphi + \frac{1}{p_{T-t}} \sum_{i,j=1}^m \partial_{y_j} (2 a^{ij}_Y
    p_{T-t}) \partial_{y_i} \varphi 
\end{align*}
so that the symmetrized generator is
\begin{align*}
  K_t \varphi 
  & := \frac{(L + \tilde{L}_{T-t})\varphi}{2} \\
  & = \frac{1}{p_t}
    \sum_{i,j=1}^n \partial_{x_j} (a^{ij}_X p_{t}) \partial_{x_i} \varphi +
    \sum_{i,j=1}^n a^{ij}_X\partial^2_{x_i x_j} \varphi +
\frac{1}{p_t}
    \sum_{i,j=1}^m \partial_{y_j} (a^{ij}_Y p_{t}) \partial_{y_i} \varphi +
    \sum_{i,j=1}^m a^{ij}_Y\partial^2_{y_i y_j} \varphi 
  \\
  &= 
  \sum_{i,j=1}^n \frac{1}{p_t} \partial_{x_i} (p_t a^{ij}_X
    \partial_{x_j} \varphi) +   \sum_{i,j=1}^m \frac{1}{p_t} \partial_{y_i} (p_t a^{ij}_Y \partial_{y_j} \varphi).
\end{align*}

For fixed $\tau \geq 0$, we see from the expression for $K$ that
$p_\tau(dx, dy)$ is an invariant measure for $K_\tau$ (use integration
by parts). 

By the properties of conditional expectation
$\int f_\tau dp_\tau = 0$. From Assumption~\ref{ass:rhoregularity} and
Theorem 1 in \autocite{pardoux_poisson_2001} it follows that
for each $\tau$ there exists a unique solution
$F_\tau \in C^2(\mathcal{X} \times \mathcal{Y})$ to the Poisson
Problem $K_\tau F_\tau = f_\tau$.

We can now apply  the forward-backward
martingale decomposition via Lemma~\ref{lemma:fbmart} to
obtain 
\begin{align*}
  \E \sup_{0\leq t\leq T}\Abs{\int_0^t f_s(X_s, Y_s) ds}^p
  &= \E \sup_{0\leq t\leq T}\Abs{\int_0^t K_s F_s(X_s, Y_s)
    ds}^p \\ 
  &= 2^{-p} \E \sup_{0\leq t\leq T}\Abs{\int_0^t (L +
    \tilde{L}_{T-s})F_s(X_s, Y_s) ds}^p \\ 
  & \leq 2^{-p} 3^{p-1} (2C_p + 1)
  {\left(\E \int_0^T 2\Gamma(F_s)(X_s, Y_s) ds \right)}^{p/2}.
\end{align*}

Now, we want to pass from $\Gamma$ to $\Gamma^X$ in order to use our
Poincaré inequality for $\nu_t^\eta$.

For $\varphi \in C^2(\mathcal{X})$ and
$y \in \mathcal{Y}, \tau \geq 0$ fixed let $\hat{K}^{\tau,y} \varphi$
be the the reversible generator associated to
$\Gamma^X(\varphi)(\cdot, y)$ and $\nu_\tau^{\phi(y)}$.

Since $\nu_\tau^{\phi(y)}$ satisfies a Poincaré inequality and $\int
f_\tau(x, y) \nu_\tau^{\phi(y)}(dx) = 0$ by assumption,
\begin{equation*}
  \hat{K}_\tau \hat{F}^{\tau,y}(x) = f_\tau(x, y)
\end{equation*}
has a unique solution $\hat{F}^{\tau,y}(x)$.

If we set $\hat{K}_\tau
\varphi(x, y) = (\hat{K}^{\tau,y} \varphi(\cdot, y))(x)$ and  $\hat{F}_\tau(x, y)
= \hat{F}^{\tau,y}(x)$ then
\begin{align*}
  \int_{\mathcal{X} \times \mathcal{Y}} \hat{K}_\tau \varphi(x, y) p_t(dx,
  dy) &= \int_{\mathcal{Y}} \int_\mathcal{X} (\hat{K}^{\tau,y} \varphi(\cdot,
  y))(x) \nu_t^{\phi(y)} p_t(\mathcal{X}, dy) = 0 \text{ and }\\
  \hat{K}_\tau \hat{F}_\tau(x, y) &= f_\tau(x, y) = K_\tau F_\tau
  (x, y).
\end{align*}

By Lemma~\ref{lemma:gamma12} we get that
\begin{equation*}
  \int_{\mathcal{X} \times \mathcal{Y}} \Gamma(F_t) dp_t
  \leq \int_{\mathcal{X} \times \mathcal{Y}} \Gamma^X(\hat{F}_t) dp_t.
\end{equation*}

Since $\hat{K}\hat{F}_t = f_t$ and $\hat{K}_t$ is the generator associated with $\Gamma^X$ and
$\nu^{\phi(y)}_t$, we can use the Poincaré inequality on $\nu^{\phi(y)}_t$ in
Lemma~\ref{lemma:gammaL} 
to estimate the right hand side by
\begin{align*}
  \int_{\mathcal{X} \times \mathcal{Y}} \Gamma^X(\hat{F}_t)(x, y) p_t(dx, dy)
  &= \int_\mathcal{Y} \int_\mathcal{X} \Gamma^X(\hat{F}_t)(x, y)
    \nu^{\phi(y)}_t(dx) p_t(\mathcal{X}, dy) \\
  & \leq \int_\mathcal{Y} c_P(\phi(y)) \int_\mathcal{X}
    {f_t}^2(x, y)
    \nu^{\phi(y)}_t(dx) p_t(\mathcal{X}, dy) \\
  &= \int_{\mathcal{X} \times \mathcal{Y}} c_P(\phi(y)) {f_t}^2(x, y)
    p_t(dx, dy)
\end{align*}

which completes the proof.

\end{proof}

\section{Distance between conditional and averaged measures}\label{sec:dist}
We will first show a general result on the relative entropy between
$\rho^{Y_t}_t$ and $\mu^{Y_t}$ by studying the relative entropy along the
trajectories of $Y_t$. We are still in the setting of section~\ref{sec:setting}.

\begin{proposition}\label{prop:supH}
  Let $f_t(x, y) = \frac{d\rho^y_t}{d\mu^y}(x)$. 
  If $\mu^y$ satisfies a Logarithmic Sobolev inequality with constant
  $c_L$ uniformly in $y$ with respect to $\Gamma^X$ then for $r \in \R$
  \begin{equation*}
    \E \HH(\rho^{Y_t}_t | \mu^{Y_t}) e^{rt}
  \leq  \E \HH(\rho^{Y_0}_0 | \mu^{Y_0}) - \left(\frac{2}{c_L}-r\right)
    \int_0^t \E \HH(\rho^{Y_s}_s | \mu^{Y_s}) e^{rs} ds + \int_0^t \E[L^Y \log f_s(X_s,
            Y_s)] e^{rs} ds.
  \end{equation*}
\end{proposition}

\begin{proof}
  We have
  \begin{equation*}
    H(\rho^y_t | \mu^{y}) = \int_\mathcal{X} f_t
    \log f_t \mu^y(dx) = \E[\log f_t(X_t, Y_t) | Y_t = y]
\end{equation*}
so that the quantity we want to estimate is
\begin{equation*}
  \E H(\rho^{Y_t}_t | \mu^{Y_t}) = \E[\log f_t(X_t, Y_t)].
\end{equation*}
Now by Itô's formula
\begin{align*}
  de^{rt} \log f_t(X_t, Y_t) 
  &= \left( (\partial_t + L) \log f_t(X_t, Y_t) + r \log f_t(X_t, Y_t)
    \right) e^{rt} dt + dM_t \\
  &= \left( (\partial_t \log \rho^y_t(x))(X_t, Y_t) + L^X \log f_t(X_t, Y_t)
    + L^Y \log f_t(X_t, Y_t) + r \log f_t(X_t, Y_t) \right)
    e^{rt} dt \\ & \quad + dM_t
\end{align*}
where $M_t$ is a local martingale.

Since $\rho^y_t dx$ is a probability measure, we have
\begin{align*}
  \E[ \partial_t \log \rho^y_t(x)(X_t, Y_t) | Y_t = y] 
  &= \int_\mathcal{X} (\partial_t \log \rho^y_t(x)) \rho^y_t(x) dx \\
  &= \int_\mathcal{X} \partial_t \rho^y_t(x) dx \\
  &= \partial_t \int_\mathcal{X} \rho^y_t(x) dx = 0.
\end{align*}

By the definition of $\mu^y$ as an invariant measure for $X^y$ we
have for all $\varphi$ in the domain of $L^X$
\begin{equation}\label{eq:muinvariant}
  \int_{{\cal X}} L^X \varphi(x, y) d\mu^y = 0.
\end{equation}
From the Logarithmic Sobolev inequality for $\mu^y$ we get
\begin{equation*}
  \HH(\rho^y_t | \mu^y) \leq \tfrac 12 {c_L} \I(\rho^y_t | \mu^y) =
  \tfrac 12 {c_L}  \int_\mathcal{X}
    \frac{\Gamma^X(f_t)(x, y)}{f_t(x)} \mu^y(x) dx.
\end{equation*}

Together with the formula
$L^X (g\circ f) = g'(f) L^X f + g''(f) \Gamma^X(f)$
this implies
\begin{align*}
  \E[L^X \log f_t(X_t, Y_t) | Y_t = y]
  &= \int_\mathcal{X} L^X (\log f_t)(x, y) \rho^y_t(x) dx \\
  &= \int_\mathcal{X} L^X f_t(x, y) \mu^y(x)dx - \int_\mathcal{X}
    \frac{\Gamma^X(f_t)(x, y)}{f_t(x)} \mu^y(x) dx \\
  &= -\I(\rho^y_t | \mu^{y}) \\
  & \leq -\frac{2}{c_L} \HH(\rho^y_t | \mu^y).
\end{align*}

By the tower property for conditional expectation and the preceding
results, $\E[ (\partial_t \log \rho^{y}_t(x) )(X_t, Y_t)] = 0$ and $\E[L^X
\log f_t(X_t, Y_t)] \leq -\frac{2}{c_L} \E  \HH(\rho^{Y_t}_t | \mu^{Y_t})$ so that
\begin{align*}
  \E \HH(\rho^{Y_t}_t | \mu^{Y_t}) e^{rt}
  &= E[\log f_t(X_t, Y_t) e^{rt}] \\
  & \leq  \E \HH(\rho^{Y_0}_0 | \mu^{Y_0}) - \left(\frac{2}{c_L}-r\right)
    \int_0^t \E \HH(\rho^{Y_s}_s | \mu^{Y_s}) e^{rs} ds + \int_0^t \E[L^Y \log f_s(X_s,
            Y_s)] e^{rs} ds.
\end{align*}

\end{proof}

We now proceed to estimate the term $\E[L^Y \log f_s(X_t,Y_t)]$ in a
restricted setting where the coefficients of $L^Y$ are independent of
$x$ and $\mu^y$ has a density $\mu^y(x) = Z(y)^{-1} e^{-V(x, y)}$
where $V$ has bounded first and second derivatives in $y$.

\begin{lemma}\label{lemma:f2mu}
  If the coefficients $b_Y^i$ and $a_Y^{ij}$ of $L^Y$ only depend on
  $y$ then for $f_t(x, y) = \frac{d\rho^y_t}{d\mu^y}(x)$
  \begin{equation*}
    \int_{\mathcal{X}} L^Y \log f_t d\rho^y_t \leq - \int_{\mathcal{X}} L^Y \log \mu^y d\rho^y_t 
  \end{equation*}
\end{lemma}
\begin{proof}
  Let $g_t(x, y) = \rho^y_t(x)$. \gaptext{Provided that all the integrals
  exist}, we have
  \begin{align*}
    \int_{\mathcal{X}} L^Y (\log g_t) (x, y) \rho^y_t(dx) 
    &= \int_{\mathcal{X}} L^Y (\log g_t(x, \cdot)) (y) \rho^y_t(dx) \\
    &= \int_{\mathcal{X}} L^Y (g_t(x, \cdot)) (y) dx -
      \int_{\mathcal{X}} \frac{\Gamma^Y (g_t(x, \cdot))(y)}{g_t(x,y)} \rho^y_t(dx) \\
    & \leq \int_{\mathcal{X}} L^Y (g_t(x, \cdot)) (y) dx \\
    & = L^Y \left(\int_{\mathcal{X}} g_t(x, \cdot) dx\right) (y) = 0
  \end{align*}
  since $g_t(x, y) dx$ is a probability measure.
  Now the result follows since 
  \begin{equation*}
        L^Y \log f_t = L^Y \log g_t - L^Y \log \mu^y.
  \end{equation*}
\end{proof}

\begin{lemma}\label{lemma:diffZ}
  Consider a probability measure $\mu(dx, dy)$ with density $\mu(x,
  y)$ on $\mathcal{X} \times
  \mathcal{Y}$ and let $Z(y) = \int_\mathcal{X} \mu(x, y) dx$,
  $\mu^y(dx) =
  \mu(dx, y) / Z(y)$.
  We have the identities
  \begin{align*}
    \partial_{y_i} \log Z(y) &= \int_\mathcal{X} \partial_{y_i} \log \mu(x, y)
                           \mu^y(dx), \\
    \partial^2_{y_i y_j} \log Z(y) &= \int_\mathcal{X} \partial^2_{y_i y_j} \log \mu(x, y) \mu^y(dx)
      + \Cov_{\mu^y}(\partial_{y_i} \log \mu, \partial_{y_j} \log \mu).
  \end{align*}
\end{lemma}
\begin{proof}
  By differentiating under the integral
  \begin{align*}
    \partial_{y_i} \log Z(y) = \frac{\partial_{y_i} Z(y)}{Z(y)} 
    &= {\int_\mathcal{X} \partial_{y_i} \mu(x, y) \frac{dx}{Z(y)}} 
      =  {\int_\mathcal{X} \frac{\partial_{y_i} \mu(x, y)}{\mu(x, y)}
      \frac{\mu(x, y) dx}{Z(y)}} 
      =  \int_\mathcal{X} \partial_{y_i} \log \mu(x, y) \mu^y(dx)
  \end{align*}
  and
  \begin{align*}
    \lefteqn{\partial^2_{y_i y_j} \log Z(y) }\quad \\
    &= \partial_{y_i} \int_\mathcal{X} \partial_{y_j} \log \mu(x, y)
      \mu^y(dx) \\
    &= \int_\mathcal{X} \partial_{y_i} \partial_{y_j} \log \mu(x, y) \mu^y(dx) 
      + \int_\mathcal{X} \partial_{y_j} \log \mu(x, y)
      \frac{\partial_{y_i} \mu(x, y)}{Z(y)} dx 
      - \int_\mathcal{X} \partial_{y_j} \log \mu(x, y) \mu(x, y)
      \frac{\partial_{y_i} Z(y)}{Z(y)^2} dx \\
    &=  \int_\mathcal{X} \partial^2_{y_i y_j} \log \mu(x, y) \mu^y(dx) 
      \\ &\quad + \int_\mathcal{X} \partial_{y_j} \log \mu(x, y)
      \partial_{y_i} \log \mu(x, y) \mu^y(dx)
      - \partial_{y_i} \log Z(y) \int_\mathcal{X} \partial_{y_j} \log
           \mu(x, y) \mu^y(dx) \\
    &=  \int_\mathcal{X} \partial^2_{y_i y_j} \log \mu(x, y) \mu^y(dx)
      + \Cov_{\mu^y}(\partial_{y_i} \log \mu, \partial_{y_j} \log \mu).
  \end{align*}
\end{proof}

\begin{lemma}\label{lemma:T2}
  For any Lipschitz function $f$
  \begin{equation*}
    \Abs{\int f d\mu^y - \int f d\rho^y_t}^2 \leq \lipnorm{f}^2 \Lambda_X c_L
    H(\rho^y_t | \mu^y)
  \end{equation*}
  uniformly in $y \in \mathcal{Y}$.
\end{lemma}
\begin{proof}
  By the Logarithmic Sobolev inequality of $\mu^y$ with respect to
  $\Gamma^X$ and the uniform boundedness of $A$ we have
  \begin{equation*}
    \Ent_{\mu^y}(f^2) \leq 2 c_L \int \Gamma^X(f) d\mu^y = 2 c_L \int
    (\nabla_x f)^T A(\cdot, y) (\nabla_x f) d\mu^y \leq 2 c_L
    \Lambda_X \int \abs{\nabla_x f}^2 d\mu^y
  \end{equation*}
  which says that $\mu^y$ satisfies a Logarithmic Sobolev inequality
  with respect to the usual square field operator $\abs{\nabla_x}^2$
  with constant $c_L \Lambda_X$. By the Otto-Villani theorem, this
  implies a $T_2$ inequality with the same constant:
  \begin{equation*}
    W_2(\rho^y_t, \mu^y)^2 \leq c_L \Lambda_X H(\rho^y_t | \mu^y).
  \end{equation*}
  By the Kantorovich duality formulation of $W_1$ and monotonicity of
  Kantorovich norms it follows from the preceding $T_2$ inequality that
  \begin{equation*}
    \Abs{\sup_{\lipnorm{f} \leq 1} \int f d(\rho^y_t - \mu^y)}^2 =
    W_1(\rho^y_t, \mu^y)^2 \leq W_2(\rho^y_t, \mu^y)^2 \leq c_L \Lambda_X H(\rho^y_t | \mu^y)
  \end{equation*}
  from which the result follows.
\end{proof}

\begin{proposition}\label{prop:LYH}
  If $b_Y$, $\sigma_Y$ depend only on $y$ and $\mu^y(dx) = Z(y)^{-1}
  e^{-V(x,y)}dx$ such that $\lipnorm{\partial_{y_i} V(\cdot, y)} <
  \infty$, $\lipnorm{\partial^2_{y_i y_j} V(\cdot, y)} < \infty$
  for all $y$ then
\begin{align*}
  \E L^Y f_t(X_t, Y_t) \leq \frac{\Lambda_X c_L}{2} \E \left(\sum_{i=1}^m
  \lipnorm{\partial_{y_i} V(\cdot, Y_t)}^2 + \sum_{i,j=1}^m \lipnorm{\partial^2_{y_i y_j} V(\cdot, Y_t)}^2\right)
  H(\rho_t^{Y_t} | \mu^{Y_t}) + \E \Phi(Y_s)
\end{align*}
where
\begin{align*}
  \Phi(y) =  
  \tfrac12 \sum_{i=1}^m b_Y^i(y)^2
  + \tfrac12 \sum_{i,j=1}^m a_Y^{ij}(y)^2
  + \sum_{i,j=1}^m a_Y^{ij}(y) \Cov_{\mu^{y}}(\partial_{y_i} V, \partial_{y_j} V).
\end{align*}
\end{proposition}

\begin{proof}
Using Lemmas~Lemma~\ref{lemma:f2mu},~\ref{lemma:diffZ} and~\ref{lemma:T2}
together with the inequality $2ab \leq a^2 + b^2$ we get
\begin{align*}
  \lefteqn{\int_\mathcal{X} L^Y \log f_t d\rho^y_t} \quad \\
  &= -\int_\mathcal{X} L^Y \log \mu^y d\rho^y_t \\
  &= L^Y \log Z(y) -\int_\mathcal{X} L^Y \log \mu d\rho^y_t \\
  &= b_Y^i (y) \int_\mathcal{X} \partial_{y_i} \log \mu\, d(\mu^y -
    \rho_t^y)
    + a_Y^{ij}(y) \int_\mathcal{X} \partial^2_{y_i y_j} \log \mu\,
    d(\mu^y - \rho^y_t)
    + a_Y^{ij}(y) \Cov_{\mu^y}(\partial_{y_i} \log \mu, \partial_{y_j}
    \log \mu) \\
  & \leq \tfrac12 b_Y^i(y)^2 + \tfrac12 \lipnorm{\partial_{y_i} \log
    \mu}^2 \Lambda_X c_L H(\rho_t^y | \mu^y)
    + \tfrac12 a_Y^{ij}(y)^2 + \tfrac12 \lipnorm{\partial^2_{y_i y_j} \log
    \mu}^2 \Lambda_X c_L
    H(\rho_t^y | \mu^y)
    \\ & \quad + a_Y^{ij}(y) \Cov_{\mu^y}(\partial_{y_i} \log \mu, \partial_{y_j}
    \log \mu).
\end{align*}
The result now follows from the tower property of conditional expectation.
\end{proof}

\section{Decoupling}\label{sec:decoupling}
We are still in the general setting of section~\ref{sec:setting}. We
also require that $\sigma_Y(x,y) = \sigma_Y(y)$ only depends on $y$
and that Assumption~\ref{ass:coeffs} is in force. The key requirement
for the results in this section is a sufficient separation of
timescales expressed by assumptions on $\gamma$.

The goal in this subsection is to estimate expressions of the type $\E
F(X, Y)$ by $\E F(\tilde{X}, Y)$ for any functional $F$ on
$\mathcal{W}_\mathcal{X} \times \mathcal{W}_\mathcal{Y}$.

Denoting $\PP$ the Wiener measure on $C([0, T], \mathcal{X} \times
\mathcal{Y})$, define a new probability measure $\QQ = \mathcal{E}(M) \PP$ with
\begin{equation*}
  dM_t = {\left({\sigma_Y(Y_t)}^{-1}(b_Y(\tilde{X}_t, Y_t) - b_Y(X_t, Y_t))\right)}^T dB^Y_t.
\end{equation*}
Corollary~\ref{corr:novikov} will show in particular that under our
assumption on $\gamma$
$\mathcal{E}(M)$ is a true martingale so that $\QQ$
is indeed a probability measure.

Under this conditions, there is a $\QQ$-Brownian motion $\tilde{B}^Y$ such that
\begin{equation*}
  dY_t = b_Y(\tilde{X}_t, Y_t)dt + \sigma_Y(Y_t) d\tilde{B}^Y_t
\end{equation*}
with
\begin{equation*}
  d\tilde{B}^Y_t = dB^Y_t - {\sigma_Y(Y_t)}^{-1}(b_Y(\tilde{X}_t, Y_t) - b_Y(X_t, Y_t))dt.
\end{equation*}

The following Proposition~\ref{prop:laws} states the key property of $\QQ$ which we are
going to use.

\begin{lemma}
  Under $\QQ$, $B^X$, $\tilde{B}^X$ and $\tilde{B}^Y$ are independent Brownian motions.
\end{lemma}
\begin{proof}
  Girsanov's theorem states that if $L$ is a continuous $\PP$-local martingale, then $L - \QV{L, M}$ is a continuous $\QQ$-local martingale. Thus $\tilde{B}^Y = B^Y - \QV{B^Y, M}$ is a continuous $\QQ$-local martingale by definition, and $B^X, \tilde{B}^X$ are continuous $\QQ$-local martingales since $\QV{B^X, M} = 0$ and $\QV{\tilde{B}^X, M} = 0$.
  Since the quadratic variation process is invariant under a change of measure we can conclude using Lévy's characterisation theorem.
\end{proof}

\begin{proposition}\label{prop:laws}
  The laws of $(X, Y, \tilde{X})$ under $\PP$ and of $(\tilde{X}, Y, X)$ under $\QQ$ are equal.
\end{proposition}
\begin{proof}
  $(X, Y)$ solves the martingale problem for $L$ under $\PP$, and
  $(\tilde{X}, Y)$ solves the martingale problem for $L$ under $\QQ$.
  Since $b_X$ and $b_Y$ are locally Lipschitz, the martingale problem
  has a unique solution.
\end{proof}

Note in particular that under $\QQ$ $B^X_t$ and $Y$ are independent.

The rest of this section is dedicated to show that we can estimate expectations under
$\PP$ by expectations under $\QQ$ when we have a sufficient separation
of timescales.

\begin{lemma}\label{lemma:EPQ}
  For any $p > 1, q > 1$ and ${\cal F}_t$-measurable variable $X$
  \begin{equation*}
    {\Bigl(\E X\Bigr)}^p \leq \Bigl(\E_\QQ X^p\Bigr)  {\Bigl(\E e^{\lambda(p, q)\QV{M}_t}\Bigr)}^{\tfrac{p-1}{q}}
    \text{ with } \lambda(p, q) = \frac{q}{2{(p-1)}^2}\left(p + \frac{1}{q-1}\right)
  \end{equation*}
\end{lemma}
\begin{proof}
  We have
  \begin{align*}
    \E X &= \E[ X {{\cal E}(M)}^{1/p} {{\cal E}(M)}^{-1/p}] \leq {(\E X^p {\cal E}(M))}^{1/p} {(\E {{\cal E}(M)}^{-p'/p})}^{1/p'}\\ & = {(\E_\QQ X^p)}^{1/p} {(\E {{\cal E}(M)}^{-p'/p})}^{1/p'} \text{ with } \tfrac{1}{p} + \tfrac{1}{p'} = 1
  \end{align*}
  Furthermore, using that for any $\alpha \in \R$ we have
  ${{\cal E}(M)}^{-\alpha} = {{\cal E}^\alpha(-M)}
  e^{\alpha(1+\alpha)/2}$, we get
  \begin{align*}
    \E[ {{\cal E}(M)}^{-p'/p} ] 
    &= \E\left[{\left( {{\cal E}(M)}^{-q'p'/p} \right)}^{1/q'}\right] 
      = \E\left[{\left( {{\cal E}^{q'p'/p}(-M)}\right)}^{1/q'} {\left(e^{\frac{q'p'}{2p}(\frac{q'p'}{p}+1)\QV{M}}\right)}^{1/q'}\right] \\
    & \leq {\left(\E{ {{\cal E}^{q'p'/p}(-M)}}\right)}^{1/q'} {\left( \E{e^{\frac{qp'}{2p}(\frac{q'p'}{p}+1)\QV{M}}}\right)}^{1/q} \text{ with }\tfrac{1}{q} + \tfrac{1}{q'} = 1 \\
    & \leq {\left(\E e^{\frac{q}{2{(p-1)}^2}\left(p + \frac{1}{q-1}\right)\QV{M}_t}\right)}^{1/q}
  \end{align*}
  The first expectation in the second line is $\leq 1$ since ${\cal E}^{q'p'/p}(-M)$ is a positive local martingale and therefore a supermartingale. Expressing $q'$ and $p'$ in terms of $p$ and $q$ in the second expectation, we pass to the last line and conclude.
\end{proof}

\begin{lemma}\label{prop:expMt}
  Under assumption~\ref{ass:coeffs} for
  \begin{equation*}
    \beta \leq \frac\gamma4
  \end{equation*}
  we have
  \begin{equation*}
   \E \exp\left({\beta \QV{M}_t}\right) \leq \exp\left({\frac{2 \beta
         \kappa_X (\alpha + n
       \bar{\lambda}_X) t}{{\Lambda_X} \gamma} }\right)
  \end{equation*}
\end{lemma}

\begin{proof}
  From the definition of $M_t$ we have
  \begin{equation*}
    d\QV{M}_t 
    = \Abs{{\sigma_Y}^{-1}\left(b(X_t, Y_t) - b(\tilde{X}_t, Y_t)\right)}^2 dt
    \leq \frac{1}{{\lambda_Y}} \Abs{b(X_t, Y_t) - b(\tilde{X}_t,
      Y_t)}^2 dt \leq \frac{\kappa_Y}{{\lambda_Y}} \Abs{X_t -
      \tilde{X}_t}^2 dt.
  \end{equation*}
We also have
\begin{align*}
  d\abs{X_t - \tilde{X}_t}^2 
  &= 2{\left(X_t - \tilde{X}_t\right)}^T \left(b_X(X_t, Y_t) - b_X(\tilde{X}_t, Y_t)\right) dt 
    \\ & + 2{\left(X_t - \tilde{X}_t\right)}^T \left(\sigma_X(X_t, Y_t)
    dB^X_t - \sigma_X(\tilde{X}_t, Y_t) d\tilde{B}^X_t\right)
    \\ & + 2 \Tr(A_X(X_t, Y_t)) dt + 2\Tr(A_X
    (\tilde{X}_t, Y_t)) dt \\
  & \leqm -2\kappa_X \abs{X_t - \tilde{X}_t}^2 dt
    + 2(\alpha + n \bar{\lambda}_X) dt
\end{align*}
where $\leqm$ means inequality modulo local
martingales,
and
\begin{align*}
  d\QV{\abs{X_t - \tilde{X}_t}^2} &= 4{(X_t -
  \tilde{X}_t)}^T \left(A_X(X_t, Y_t) +
  A_X(\tilde{X}_t, Y_t) \right) (X_t -
  \tilde{X}_t) \\ & \leq 8 {\Lambda_X} \abs{X_t - \tilde{X}_t}^2
\end{align*}
so that
\begin{align*}
  de^{\frac r 2 \abs{X_t - \tilde{X}_t}^2} e^{\beta \QV{M}_t} 
  &= \left(\frac r 2 d\abs{X_t - \tilde{X}_t}^2 + \beta d\QV{M}_t +
    \frac{r^2}{8} d\QV{\abs{X_t - \tilde{X}_t}^2} \right)
    e^{\frac r 2 \abs{X_t - \tilde{X}_t}^2} e^{\beta \QV{M}_t} \\
  & \leqm \left( \left(r^2 {\Lambda_X} - r\kappa_X +
    \frac{\beta \kappa_Y^2}{\lambda_Y^2} \right) \abs{X_t - \tilde{X}_t}^2 +
    r(\alpha + n \bar{\lambda}_X) \right) e^{\frac r 2 \abs{X_t -
    \tilde{X}_t}^2} e^{\beta \QV{M}_t} dt \\
  &= \left( {\Lambda_X}(r-r_-)(r-r_+) \abs{X_t - \tilde{X}_t}^2 +
    r(\alpha + n \bar{\lambda}_X) \right) e^{\frac r 2 \abs{X_t -
    \tilde{X}_t}^2} e^{\beta \QV{M}_t} dt
\end{align*}
with
\begin{equation*}
  r_{\pm} = \frac{\kappa_X}{2{\Lambda_X}}\left(1\pm\sqrt{1-4\beta/\gamma} \right).
\end{equation*}
According to our assumptions, $1-4\beta/\gamma > 0$
and we have, choosing $r=r_-$
\begin{align*}
  de^{\frac {r_-} 2 \abs{X_t - \tilde{X}_t}^2} e^{\beta \QV{M}_t} 
  & \leqm r_-(\alpha + n \bar{\lambda}_X) e^{\frac {r_-} 2 \abs{X_t -
    \tilde{X}_t}^2} e^{\beta \QV{M}_t} dt
\end{align*}
so that
\begin{equation*}
  e^{\frac {r_-} 2 \abs{X_t - \tilde{X}_t}^2} e^{\beta \QV{M}_t} \leqm
  e^{r_-(\alpha + n \bar{\lambda}_X) t}
\end{equation*}
and
\begin{equation*}
  \E e^{\beta \QV{M}_t} \leq \E e^{\frac {r_-} 2 \abs{X_t -
      \tilde{X}_t}^2} e^{\beta \QV{M}_t} \leq  e^{r_-(\alpha + n \bar{\lambda}_X) t}.
\end{equation*}
Since $1 - \sqrt{1-x} \leq x$ for $0 \leq x \leq 1$ we have furthermore
\begin{equation*}
  r_- \leq \frac{\kappa_X}{2{\Lambda_X}} \frac{4\beta}{\gamma}
\end{equation*}
so that
\begin{equation*}
   \E e^{\beta \QV{M}_t} \leq  \exp\left({\frac{2 \beta
         \kappa_X (\alpha + n
       \bar{\lambda}_X) t}{{\Lambda_X} \gamma} }\right).
\end{equation*}

\end{proof}

\begin{corollary}\label{corr:novikov}
  If $\gamma > 2$ then
  \begin{equation*}
    {\mathcal{E}(M)}_t \text{ is a true martingale.}
  \end{equation*}
\end{corollary}
\begin{proof}
  Since $\tfrac{1}{2} < \tfrac\gamma4$ by our assumption we get from the previous Proposition that
\begin{equation*}
\E\left[e^{\frac{1}{2}\QV{M}_t}\right] < \infty
\end{equation*}
and Novikov's criterion leads directly to the conclusion.
\end{proof}

\begin{proposition}\label{prop:20}
  Under assumption~\ref{ass:coeffs} for any $\mathcal{F}_t$-measurable
  random variable $Z$ and
  $1+\tfrac2\gamma + 2\sqrt{\tfrac2\gamma} \leq p \leq 2$
  \begin{equation*}
    {\left(\E Z\right)}^p \leq \E_\QQ\left[{Z}^p\right] \, \exp\left({\frac{p\,
    \kappa_X (\alpha + n
    \bar{\lambda}_X) \, t}{\left(p-1-\sqrt{2/\gamma}\right)\,
    {\Lambda_X} \, \gamma }}\right)
  \end{equation*}
\end{proposition}
\begin{proof}
  We would like to apply Lemmas~\ref{lemma:EPQ} and~\ref{prop:expMt},
  so we need to find conditions that ensure the existence of a $q$
  such that $\lambda(p, q) \leq \tfrac\gamma4$.

After some straightforward computations we get the identities
\begin{align*}
  \lambda(p, q) - \frac\gamma4 &= \frac{p(q-q_-)(q-q_+)}{2{(p-1)}^2(q-1)}, \\
  q_{\pm} &= \frac{\gamma (p-1)}{4p} \left(p - 1 + \frac2\gamma \pm
            \sqrt{(p-p_-)(p-p_+)} \right), \\
  p_\pm &= 1+\tfrac2\gamma \pm 2\sqrt{\tfrac2\gamma}.
\end{align*}

Our assumption on $p$ implies that $1+\tfrac2\gamma +
2\sqrt{\tfrac2\gamma} \leq 2 \iff \gamma \geq \frac{1}{(\sqrt3 -
  \sqrt2)^2} > 2$
so that
$p - p_- > p - 1 + \tfrac{2}{\gamma} > 0$ and by our assumption on $p$,
$p-p_+ > 0$ as well so that $q_\pm$ is real and $\lambda(p, q_+) = \frac\gamma4$.

For our particular values of $p_-$ and $p_+$ we have furthermore $(p-p_-)(p-p_+) \geq {(p-p_+)}^2$ so that
\begin{equation*}
  q_+ \geq \frac{\gamma (p-1)(p-1-\sqrt{\tfrac2\gamma})}{2p}
\end{equation*}

Now, apply Lemma~\ref{lemma:EPQ} with $q = q_+$ to obtain
\begin{equation*}
  {\E[Z]}^p \leq \E_\QQ\bigl[Z^p\bigr]{\E\bigl[e^{\tfrac{\gamma}{4} \QV{M}_t}\bigr]}^{\tfrac{p-1}{q_+}}.
\end{equation*}
We estimate the second expectation on the right hand side using
Proposition~\ref{prop:expMt}
\begin{align*}
  {\E\bigl[e^{\tfrac{\gamma}{4} \QV{M}_t}\bigr]}^{\tfrac{p-1}{q_+}} 
  &\leq \exp\left(\frac{(p-1)}{q_+} {\frac{\kappa_X (\alpha + n
    \bar{\lambda}_X) t}{2 {\Lambda_X}} }\right) \\
  &\leq \exp\left({\frac{p\,
    \kappa_X (\alpha + n
    \bar{\lambda}_X) \, t}{\left(p-1-\sqrt{2/\gamma}\right)\,
    {\Lambda_X} \, \gamma }}\right)
\end{align*}
which leads to our result.
\end{proof}

\section{Proof of the main theorem}\label{sec:mainthm}

\begin{lemma}\label{prop:10}
  If $\bar{b}_Y$ is Lipschitz then
  \begin{equation*}
    \sup_{0\leq t \leq T} \abs{Y_t - \bar{Y}_t} \leq \sup_{0\leq t\leq T}\Abs{\int_0^t b_Y(X_s, Y_s) - \bar{b}(Y_s) ds} e^{\lipnorm{\bar{b}} T}
  \end{equation*}
\end{lemma}

\begin{proof}
  \begin{align*}
    \sup_{0\leq t\leq T} \abs{Y_t - \bar{Y}_t}  
    &= \sup_{0\leq t\leq T} \Abs{\int_0^t b_Y(X_s, Y_s) - \bar{b}_Y(\bar{Y}_s) ds} \\
    &\leq \sup_{0\leq t\leq T} \Abs{\int_0^t b_Y(X_s, Y_s) - \bar{b}_Y(Y_s) ds} 
      + \lipnorm{\bar{b}} \int_0^T \sup_{0\leq s\leq t} \abs{Y_s - \bar{Y}_s} ds
  \end{align*}
  and the conclusion follows from Gronwall's inequality.
\end{proof}

\begin{theorem}\label{theorem:1}
  Under Assumption~\ref{ass:coeffs} if $\sigma_Y(x, y) = \sigma_Y(y)$,
  a Poincaré inequality with constant $c_P$ holds for
  $\tilde{\rho}^y_t$, a Logarithmic Sobolev
  inequality with constant $c_L$ holds for $\mu^y(dx) = Z(y)^{-1}
  e^{-V(x, y)} dx$ both with respect
  to $\Gamma^X$, $X_0 \sim \mu^{Y_0}$
  and $\bar{b}$ is Lipschitz 
  then for $1 \leq p \leq \tfrac{2}{1+\tfrac{2}{\gamma} +
    2\sqrt{\tfrac{2}{\gamma}}}$ we have the estimate
  \begin{multline*}
    {\E\left[\sup_{0\leq t\leq T}\abs{Y_t - \bar{Y}_t}^p\right]}^{2/p} \leq
    {m {\kappa_Y}^2} {\Lambda_X} \left( 27 {c_P}^2 T 
      + \frac{2 {c_L}^2}{4-{c_L}^2\Lambda_X({m {\kappa_Y}^2} + 3 {c_V}^2)} \E \int_0^T
    \Psi(Y_t) dt \right) \\ \exp\left(\frac{2 p' \kappa_X (\alpha +
    n\bar{\lambda}_X)T}{p \gamma \Lambda_X} + 2 \lipnorm{\bar{b}} T\right)
  \end{multline*}
with
\begin{equation*}
  \Psi(y) =  \frac{3 {m {\kappa_Y}^2}
      (\alpha + n\bar{\lambda}_X)}{2\kappa_X} + \tfrac32 
    \abs{\bar{b}(y)}^2 + \tfrac12 \sum_{i,j=1}^m a_Y^{ij}(y)^2
    \\
    + \sum_{i,j=1}^m a_Y^{ij}(y) \Cov_{\mu^y}(\partial_{y_i}
    V, \partial_{y_j} V),
\end{equation*}
\begin{equation*}
  p' = \frac{1}{1-\tfrac{p}{2}\left(1+ \sqrt{\tfrac 2 \gamma}\right)}
  >   \frac{2}{2-p}
\end{equation*}
and
\begin{equation*}  {c_V}^2 = \sup_y 
  \left(\sum_{i=1}^m
  \lipnorm{\partial_{y_i} V(\cdot, y)}^2 + \sum_{i,j=1}^m
  \lipnorm{\partial^2_{y_i y_j} V(\cdot, y)}^2\right).
\end{equation*}

\end{theorem}
\begin{proof}
  By Lemma~\ref{prop:10} we have
  \begin{equation*}
    \E\left[\sup_{0\leq t\leq T}\abs{Y_t - \bar{Y}_t}^p\right] \leq
    \E\left[ \sup_{0\leq t\leq T}\Abs{\int_0^t b_Y(X_s, Y_s) -
        \bar{b}(Y_s) ds}^p\right] e^{p\lipnorm{\bar{b}} T}.
  \end{equation*}
  Using Proposition~\ref{prop:20} we get for
  $1 \leq p \leq \tfrac{2}{1+\tfrac{2}{\gamma} +
    2\sqrt{\tfrac{2}{\gamma}}}$ that
  \begin{multline*}
    \E\left[ \sup_{0\leq t\leq T}\Abs{\int_0^t b_Y(X_s, Y_s) -
        \bar{b}(Y_s) ds}^p\right] \\ \leq {\E_\QQ\left[ \sup_{0\leq t\leq T}\Abs{\int_0^t b_Y(X_s, Y_s) -
        \bar{b}(Y_s) ds}^2\right]}^{p/2} \exp\left(\frac{p' \kappa_X (\alpha +
    n\bar{\lambda_X})T}{\gamma \Lambda_X}\right)
  \end{multline*}
  with
  \begin{equation*}
    0 < p' = \frac{1}{1-\tfrac{p}{2}\left(1+ \sqrt{\tfrac 2 \gamma}\right)} < \infty.
  \end{equation*}
  By Proposition~\ref{prop:laws}
  \begin{equation*}
    \E_\QQ\left[ \sup_{0\leq t\leq T}\Abs{\int_0^t b_Y(X_s, Y_s) -
        \bar{b}(Y_s) ds}^2\right] = \E\left[ \sup_{0\leq t\leq T}\Abs{\int_0^t b_Y(\tilde{X}_s, Y_s) -
        \bar{b}(Y_s) ds}^2\right].
  \end{equation*}
  Now we decompose
  \begin{multline}\label{eq:decomposition}
    {\E\left[ \sup_{0\leq t\leq T}\Abs{\int_0^t b_Y(\tilde{X}_s, Y_s) -
        \bar{b}(Y_s) ds}^2\right]}  \leq {2\E\left[ \sup_{0\leq t\leq T}\Abs{\int_0^t b_Y(\tilde{X}_s, Y_s) -
       \E[b_Y(\tilde{X}_s, Y_s) | (X_s, Y_s)] ds}^2\right]} \\ +  2
 {\E\left[ \sup_{0\leq t\leq T}\Abs{\int_0^t \E[b_Y(\tilde{X}_s,
       Y_s) | (X_s, Y_s)] -
        \bar{b}(Y_s) ds}^2\right]}.
  \end{multline}
  For the rest of the proof we put ourselves in
  the setting of section~\ref{sec:setting} where we substitute
  $\tilde{X}$ for $X$ and $(X, Y)$ for $Y$.

  For $1 \leq i \leq m$ we now apply Proposition~\ref{prop:sup} with $\phi: (x, y)
  \mapsto y$, $\nu^y_t = \tilde{\rho}^y_t$ and $f_t(\tilde{x}, x, y) = b^i_Y(\tilde{x}, y) -
       \E[b^i_Y(\tilde{X}_s, Y_s) | (X_s, Y_s) = (x, y)]$. Since
       $\tilde{\rho}^y_t$ satisfies a Poincaré inequality by
       assumption and $\int f_t(\cdot, y) d\tilde{\rho}^y_t = 0$ by the
       properties of conditional expectation, we get
  \begin{align*}
    \E\left[ \sup_{0\leq t\leq T}\Abs{\int_0^t b^i_Y(\tilde{X}_s, Y_s) -
       \E[b^i_Y(\tilde{X}_s, Y_s) | (X_s, Y_s)] ds}^2\right] 
& \leq \frac{27}{2} 
  \int_0^T \E\left[ {c_P(Y_t)} f_t^2(\tilde{X}_t, X_t, Y_t)
  \right] dt  \\
&  \leq \frac{27}{2} 
  \int_0^T \E\left[ {c_P(Y_t)}^2 \Gamma^X(b^i_Y)(\tilde{X}_t, Y_t)
  \right] dt   \\
& \leq \frac{27  {c_P}^2 \Lambda_X  \norm{{\nabla_x b^i_Y}}^2_\infty T}{2} 
  \end{align*}
  where the second inequality follows from the tower property of
  conditional expectation and applying the Poincaré
  inequality a second time to $\tilde{\rho}^y_t$ and the last line
  from $\Gamma^X(b^i_Y) = {\nabla_x
    b^i_Y}^T A_X \nabla_x b^i_Y \leq \Lambda_X \abs{\nabla_x{b^i_Y}}^2 \leq
  \Lambda_X \norm{\nabla_x{b^i_Y}}_\infty^2$.
  Summing over the components $b_Y^i$ we get
  \begin{multline*}
    \E\left[ \sup_{0\leq t\leq T}\Abs{\int_0^t b_Y(\tilde{X}_s, Y_s) -
       \E[b_Y(\tilde{X}_s, Y_s) | (X_s, Y_s)] ds}^2\right] \leq
   \frac{27  {c_P}^2 \Lambda_X T}{2}    \sum_{i=1}^m  \norm{{\nabla_x
       b^i_Y}}^2_\infty \\
   =    \frac{27  {c_P}^2 \Lambda_X m {\kappa_Y}^2 T}{2}
  \end{multline*}

  We now turn to the second term on the right hand side in the 
  decomposition~\eqref{eq:decomposition}.
  First, note that
  \begin{align*}
\lefteqn{    {\E\left[ \sup_{0\leq t\leq T}\Abs{\int_0^t \E[b_Y(\tilde{X}_s,
    Y_s) | (X_s, Y_s)] -
        \bar{b}(Y_s) ds}^2\right]} }\quad  \\ & \leq \E \int_0^T \Abs{\E[b_Y(\tilde{X}_s,
    Y_s) | (X_s, Y_s)] -
        \bar{b}(Y_s) ds}^2 \\
      &= \sum_{i=1}^n \int_0^T \E \Abs{\int_\mathcal{X} b^i_Y(\tilde{x}, Y_t)
        \tilde{\rho}^{X_t,Y_t}(d\tilde{x}) - \int_\mathcal{X}
        b^i_Y(\tilde{x}, Y_t) \mu^{Y_t}(d\tilde{x})}^2
\end{align*}
By Lemma~\ref{lemma:T2} we have
\begin{multline*}
\sum_{i=1}^n \E \Abs{\int_\mathcal{X} b^i_Y(\tilde{x}, Y_t)
        \tilde{\rho}^{X_t,Y_t}(d\tilde{x}) - \int_\mathcal{X}
        b^i_Y(\tilde{x}, Y_t) \mu^{Y_t}(d\tilde{x})}^2 \\ \leq
  \sum_{i=1}^n \lipnorm{b_Y^i}^2 \Lambda_X c_L
    \E H(\tilde{\rho}^{X_t,Y_t}_t | \mu^{Y_t}) \leq {m {\kappa_Y}^2} \Lambda_X c_L
    \E H(\tilde{\rho}^{X_t,Y_t}_t | \mu^{Y_t}).
\end{multline*}

Suppose that uniformly in $y$
\begin{equation*}
  \left(\sum_{i=1}^m
  \lipnorm{\partial_{y_i} V(\cdot, y)}^2 + \sum_{i,j=1}^m
  \lipnorm{\partial^2_{y_i y_j} V(\cdot, y)}^2\right) < {c_V}^2.
\end{equation*}
Now, for some $r \in \R$ to be fixed later, use
Propositions~\ref{prop:supH} and~\ref{prop:LYH} to get
\begin{align}\label{eq:Ht}
  \E H(\tilde{\rho}^{X_t,Y_t}_t | \mu^{Y_t}) e^{rt} 
  & \leq -\left(\frac{2}{c_L} - \frac{{c_V}^2 \Lambda_X c_L}{2} -
    r\right) \int_0^t \E H(\rho_s^{X_s, Y_s} | \mu^{Y_s}) e^{rs}ds +
    \int_0^t \E \Phi(X_s, Y_s) e^{rs}ds.
\end{align}

We have
\begin{equation*}
  \E \Phi(X_s, Y_s)
  = \E \left[ \tfrac12 \sum_{i=1}^m b_Y^i(X_s, Y_s)^2
  + \tfrac12 \sum_{i,j=1}^m a_Y^{ij}(Y_s)^2
  + \sum_{i,j=1}^m a_Y^{ij}(Y_s) \Cov_{\mu^{Y_s}}(\partial_{y_i}
  V, \partial_{y_j} V) \right]
\end{equation*} and we estimate the first term on the right hand side as
follows:
\begin{multline*}
  \E b_Y^i(X_s, Y_s)^2 \leq 3 \E [b_Y^i(X_s, Y_s) - b_Y^i(\tilde{X}_s,
  Y_s)]^2 \\ + 3 [\E b_Y^i(\tilde{X}_s,
  Y_s) - \int_\mathcal{X} b^i_Y(x, Y_s) \mu^{Y_s}(dx)]^2 +
  3 \Abs{\int_\mathcal{X} b^i_Y(x, Y_s) \mu^{Y_s}(dx)}^2.
\end{multline*}

Since $b_Y$ is Lipschitz in the first variable we get for the first term
\begin{equation*}
  \sum_{i=1}^m \E [b_Y^i(X_s, Y_s) - b_Y^i(\tilde{X}_s,
  Y_s)]^2 = \E \Abs{b_Y(X_s, Y_s) - b_Y(\tilde{X}_s, Y_s)}^2 \leq
  {m {\kappa_Y}^2} \E\abs{X_s - \tilde{X}_s}^2 \leq \frac{{m {\kappa_Y}^2} (\alpha + n\bar{\lambda}_X)}{\kappa_X}.
\end{equation*}
Still using the Lipschitzness of $b_Y$, we use Lemma~\ref{lemma:T2} together
with the tower property for conditional expectation on the
second term to get
\begin{equation*}
   \sum_{i=1}^m [\E b_Y^i(\tilde{X}_s,
  Y_s) - \int_\mathcal{X} b^i_Y(x, Y_s) \mu^{Y_s}(dx)]^2 \leq
  {m {\kappa_Y}^2} c_L \Lambda_X \E \HH(\tilde{\rho}^{X_s,Y_s}_s | \mu^{Y_s}).
\end{equation*}
This leads us to
\begin{multline*}
  \E \Phi(X_s, Y_s) \leq \tfrac32 \left( {m {\kappa_Y}^2} c_L \Lambda_X \E
    \HH(\tilde{\rho}^{X_s,Y_s}_s | \mu^{Y_s}) + \frac{{m {\kappa_Y}^2}
      (\alpha + n\bar{\lambda}_X)}{\kappa_X} + \E \abs{\bar{b}(Y_s)}^2
  \right) \\
  + \tfrac12 (a_Y^{ij}(Y_s) 
  + a_Y^{ij}(Y_s) \Cov_{\mu^{Y_s}}(\partial_{y_i} V, \partial_{y_j} V).
\end{multline*}
Substituting $\Phi$ in~\eqref{eq:Ht} we get
\begin{multline*}
  \E H(\tilde{\rho}^{X_t,Y_t}_t | \mu^{Y_t}) e^{rt}
  \leq -\left(\frac{2}{c_L} - \frac{\Lambda_X c_L ({m {\kappa_Y}^2}+3{c_V}^2)}{2} -
    r\right) \int_0^t \E H(\rho_s^{X_s, Y_s} | \mu^{Y_s}) e^{rs}ds \\
  +  \E \int_0^t e^{rs} \frac{ {3 m \kappa_Y}^2
      (\alpha + n\bar{\lambda}_X)}{2\kappa_X} +
    \tfrac32 \abs{\bar{b}(Y_s)}^2
    + \tfrac12 \sum_{i,j=1}^m a_Y^{ij}(Y_s)^2 \\
    + \sum_{i,j=1}^m a_Y^{ij}(Y_s) \Cov_{\mu^{Y_s}}(\partial_{y_i}
    V, \partial_{y_j} V) ds.
\end{multline*}

Now we choose
\begin{equation*}
  r = \frac{2}{c_L} - \frac{\Lambda_X c_L ({m {\kappa_Y}^2}+3{c_V}^2)}{2}
\end{equation*}
so that
\begin{equation*}
  \E H(\tilde{\rho}^{X_t,Y_t}_t | \mu^{Y_t}) \leq \E \int_0^t
  e^{-r(t-s)} \Psi(Y_s) ds.
\end{equation*}
with
\begin{equation*}
  \Psi(y) =  \frac{3 {m {\kappa_Y}^2}
      (\alpha + n\bar{\lambda}_X)}{2\kappa_X} + \tfrac32 
    \abs{\bar{b}(y)}^2 + \tfrac12 \sum_{i,j=1}^m a_Y^{ij}(y)^2
    \\
    + \sum_{i,j=1}^m a_Y^{ij}(y) \Cov_{\mu^y}(\partial_{y_i}
    V, \partial_{y_j} V).
\end{equation*}

By the preceding inequality and the Young inequality for
convolutions on $L^1([0, T])$
\begin{align*}
  \int_0^T \E H(\tilde{\rho}^{X_t,Y_t}_t | \mu^{Y_t}) dt
 & \leq \E \int_0^T \int_0^t e^{-r(t-s)} \Psi(Y_s) ds dt \\
 & \leq \E \int_0^T e^{-rt} dt \int_0^T \Psi(Y_t) dt \\
 & = \frac{1}{r} (1-e^{-rT}) \E \int_0^T \Psi(Y_t) dt
\end{align*}
so that finally
\begin{align*}
   \E\left[ \sup_{0\leq t\leq T}\Abs{\int_0^t \E[b_Y(\tilde{X}_s,
    Y_s) | (X_s, Y_s)] -
        \bar{b}(Y_s) ds}^2\right]
  &\leq \frac{{m {\kappa_Y}^2} \Lambda_X c_L}{r} (1-e^{-rT}) \E \int_0^T
    \Psi(Y_t) dt \\
  &= \frac{2 {c_L}^2 \Lambda_X
    {m {\kappa_Y}^2}}{4-{c_L}^2\Lambda_X({m {\kappa_Y}^2} + 3 {c_V}^2)} (1-e^{-rT}) \E \int_0^T
    \Psi(Y_t) dt \\
  &\leq \frac{2 {c_L}^2 \Lambda_X
    {m {\kappa_Y}^2}}{4-{c_L}^2\Lambda_X({m {\kappa_Y}^2} + 3 {c_V}^2)} \E \int_0^T
    \Psi(Y_t) dt.
\end{align*}
  Assembling the previous results, we obtain
  \begin{multline*}
    {\E\left[\sup_{0\leq t\leq T}\abs{Y_t - \bar{Y}_t}^p\right]}^{2/p} \leq
    {m {\kappa_Y}^2} {\Lambda_X} \left( 27 {c_P}^2 T 
      + \frac{2 {c_L}^2}{4-{c_L}^2\Lambda_X({m {\kappa_Y}^2} + 3 {c_V}^2)} \E \int_0^T
    \Psi(Y_t) dt \right) \\ \exp\left(\frac{2 p' \kappa_X (\alpha +
    n\bar{\lambda}_X)T}{p \gamma \Lambda_X} + 2 \lipnorm{\bar{b}} T\right).
  \end{multline*}
\end{proof}

\section{Applications}\label{sec:examples}

\subsection{Averaging}

For $\varepsilon > 0$ fixed consider an SDE of the form 
\begin{align}\label{eq:avg}
  dX_t &= -\varepsilon^{-1} \nabla_x V(X_t, Y_t) dt +
         {\varepsilon}^{-1/2} \sqrt{2 {\beta_X^{-1}}} dB^X_t \\
  dY_t &= b_Y(X_t, Y_t) dt +  \sqrt{2
  \beta_Y^{-1}} dB^Y_t
\end{align}
with $Y_0 = y_0 \in \R^m$ and $X_0 \sim \mu^{y_0} = e^{-\beta V(x,
  y_0)} dx$ and $V(x, y)$ is of the form
\begin{equation*}
  V(x, y) = \frac12 (x-g(y))Q(x-g(y)) + h(x, y)
\end{equation*}
where $h$ is uniformly bounded in both arguments and both $\partial_y h$ and
$\partial^2_y h$ are Lipschitz in $x$ uniformly in $y$. Under these
conditions
\begin{equation*}
\mu^y(dx) = Z(y)^{-1} e^{-\beta_X V(x, y)}dx \text{ with } Z(y) =
\int_\mathcal{X} e^{-\beta_X V(x, y)} dx 
\end{equation*}
is a Gaussian
measure with covariance matrix $\beta_X Q$ and mean $g(y)$ perturbed by
a bounded factor $e^{-\beta_X h(x, y)}$. As such it satisfies a
Logarithmic Sobolev inequality with respect to the usual square field
operator $\abs{\nabla}^2$ with constant
\begin{equation*}
  c_L^0 = (\beta_X \lambda_Q)^{-1} e^{\beta_X \osc(h)} \quad \text{with } \osc(h) = \sup h - \inf h
\end{equation*}
and $\lambda_Q$ is the smallest eigenvalue of $Q$. 
In particular,
$\mu^y$ satisfies a Logarithmic Sobolev inequality with constant
\begin{equation*}
  c_L = \varepsilon \lambda_Q^{-1} e^{\beta_X \osc(h)}
\end{equation*}
with respect to $\Gamma^X = \varepsilon^{-1} \beta_X^{-1} \abs{\nabla}^2$.

We have
\begin{align*}
\lefteqn{  -(x_1 - x_2)^T (\nabla_x V(x_1, y) - \nabla_x V(x_2, y))
  }\quad \\
&= -(x_1 - x_2)^T Q (x_1 - x_2) - (x_1 - x_2)^T (\nabla_x h(x_1, y) -
  \nabla_x h(x_2, y)) \\
  & \leq -\lambda_Q \abs{x_1 - x_2}^2 + \abs{x_1 - x_2} \norm{\nabla_x
    h}_\infty \\
  & \leq -\lambda_Q \abs{x_1 - x_2}^2 + \frac{\norm{\nabla_x
    h}_\infty}{4 \lambda_Q}
\end{align*}
so that we can choose
\begin{align*}
  \kappa_X &= \varepsilon^{-1} \lambda_Q, & \alpha = \varepsilon^{-1} \frac{\norm{\nabla_x
    h}_\infty}{4 \lambda_Q}.
\end{align*}
We also have trivially
\begin{align*}
  \lambda_X &= \Lambda_X = \bar{\lambda}_X = \varepsilon^{-1}
              \beta_X^{-1}, & \Lambda_Y &= \beta_Y^{-1}, & \kappa_Y &=
                                                                      \norm{\nabla_x b_Y}_\infty
\end{align*}
and the separation of timescales is
\begin{equation*}
  \gamma = \frac{{\kappa_X}^2{\lambda_Y}}{{\Lambda_X}{m {\kappa_Y}^2}} =
  \varepsilon^{-1} \frac{{\lambda_Q}^2 \beta_Y^{-1}}{\norm{\nabla_x
      b_Y}^2_\infty \beta_X^{-1}}.
\end{equation*}

If $\gamma > \frac{1}{(\sqrt3 - \sqrt2)^2} \approx 9.899$ we can
apply Theorem~\ref{theorem:1} with $p=1$ to get
  \begin{multline*}
    {\E\left[\sup_{0\leq t\leq T}\abs{Y_t - \bar{Y}_t}\right]}^2 \leq
    \varepsilon C_1 \left( 27 {(c_P(\varepsilon)/c_L)}^2 T 
      + C_2 \E \int_0^T
    \Psi(Y_t) dt \right) \exp\left(2 p' C_3 T + 2 \lipnorm{\bar{b}} T\right)
  \end{multline*}

with

\begin{align*}
  C_1 &= \varepsilon^{-1} {m {\kappa_Y}^2} {\Lambda_X} {c_L}^2 =
  {m {\kappa_Y}^2} \beta_X^{-1} \lambda_Q^{-2} e^{2\beta_X \osc(h)}, \\
  C_2 &= \frac{2}{4-{c_L}^2\Lambda_X({m {\kappa_Y}^2} + 3 {c_V}^2)} \leq 1
  \text{ for } \varepsilon \leq \frac{2\lambda_Q e^{-\beta_X
      \osc(h)}\beta_X}{\norm{\nabla_X b_Y}^2_\infty + 3 {c_V}^2}, \\
  C_3 &= \frac{\kappa_X (\alpha +
    n\bar{\lambda}_X)}{\gamma \Lambda_X} = \frac{\norm{\nabla_x
      b_Y}_\infty^2 (\frac{\norm{\nabla_x h}_\infty}{4\lambda_Q} + n
    \beta^{-1}_X)}{ \beta_Y^{-1} \lambda_Q},
\end{align*}
\begin{align*}
  \Psi(y) &=  \frac{3 {m {\kappa_Y}^2}
      (\alpha + n\bar{\lambda}_X)}{2\kappa_X} + \tfrac32 
    \abs{\bar{b}(y)}^2 + \tfrac12 \sum_{i,j=1}^m a_Y^{ij}(y)^2
    + \sum_{i,j=1}^m a_Y^{ij}(y) \Cov_{\mu^y}(\partial_{y_i}
    \beta_X V, \partial_{y_j} \beta_X V) \\
  &= \frac{3 \norm{\nabla_x b_Y}_\infty^2 (\frac{\norm{\nabla_x h}_\infty}{4\lambda_Q} + n \beta^{-1}_X)}{2 \lambda_Q} + \tfrac12
    m \beta_Y^{-2} + \beta_Y^{-1} \beta_X^2 \sum_i \Var_{\mu^y}
    (\partial_{y_i} V) +  \tfrac32 
    \abs{\bar{b}(y)}^2 \\
  & \leq  \frac{3 \norm{\nabla_x b_Y}_\infty^2 (\frac{\norm{\nabla_x h}_\infty}{4\lambda_Q} + n \beta^{-1}_X)}{2 \lambda_Q}+ \tfrac12
    m \beta_Y^{-2} + \beta_Y^{-1} \beta_X^2 c_L^0 \sum_i
    \lipnorm{\partial_{y_i} V}^2 +  \tfrac32 
    \abs{\bar{b}(y)}^2 \\
  & =  \frac{3 \norm{\nabla_x b_Y}_\infty^2 (\frac{\norm{\nabla_x h}_\infty}{4\lambda_Q} + n \beta^{-1}_X)}{2 \lambda_Q} + \tfrac12
    m \beta_Y^{-2} + \beta_X (\beta_Y \lambda_Q)^{-1} e^{\beta_X
    \osc(h)} \sum_i
    \lipnorm{\partial_{y_i} V}^2 +  \tfrac32 
    \abs{\bar{b}(y)}^2,
\end{align*}
\begin{equation*}
  2 < p' = \frac{1}{1-\tfrac{1}{2}\left(1+ \sqrt{\tfrac 2 \gamma}\right)}
  < \frac{2}{3-\sqrt2\sqrt3} \approx 3.633
\end{equation*}
and
\begin{equation*}  {c_V}^2 = \sup_y 
  \left(\sum_{i=1}^m
  \lipnorm{\partial_{y_i} V(\cdot, y)}^2 + \sum_{i,j=1}^m
  \lipnorm{\partial^2_{y_i y_j} V(\cdot, y)}^2\right).
\end{equation*}

If we suppose that $c_P(\varepsilon)/{c_L}$ converges to a finite
limit as $\varepsilon \to 0$ and that
\begin{equation*}
\E \int_0^T \bar{b}(Y_t)^2 dt  < \infty
\end{equation*}
then there exists a constant $C$ depending on $T, V, \beta_X,
b_Y$ and $\beta_Y$ such that for $\varepsilon$ sufficiently small
\begin{equation*}
  \E{\sup_{0\leq t \leq T} \abs{Y_t - \bar{Y}_t}} \leq \sqrt{\varepsilon} C.
\end{equation*}
In other words, we obtain a strong averaging principle of order $1/2$ in $\varepsilon$.

\subsection{Temperature-Accelerated Molecular Dynamics}

In \autocite{maragliano_temperature_2006,} the authors introduced the
TAMD process $(X_t, Y_t)$ and its averaged version $\bar{Y}_t$ defined
by
\begin{align*}
  dX_t &= -\tfrac1\varepsilon \nabla_x U(X_t, Y_t) dt + \sqrt{2
  {(\beta\varepsilon)}^{-1}} dB^X_t, \quad X_0 \sim e^{-\beta U(x, y_0)} dx \\
  dY_t &= -\tfrac1{\bar{\gamma}}\kappa (Y_t - \theta(X_t))dt +  \sqrt{2
  {(\bar{\beta} \bar{\gamma})}^{-1}} dB^Y_t, \quad Y_0 = y_0\\
  d\bar{Y}_t &= \bar{b}(\bar{Y}_t) dt + \sqrt{2
  {(\bar{\beta} \bar{\gamma})}^{-1}} dB^Y_t, \quad \bar{Y}_0 = y_0 \\
  U(x, y) &= V(x) + \tfrac{\kappa}{2}\abs{y-\theta(x)}^2, \\
  \bar{b}(y) &= Z(y)^{-1} \int -\bar{\gamma}^{-1}\kappa (y - \theta(x))
               e^{-\tfrac{\kappa}{2}\abs{y-\theta(x)}^2} e^{-V(x)} dx,
               \quad Z(y) = \int e^{-\tfrac{\kappa}{2}\abs{y-\theta(x)}^2} e^{-V(x)} dx
\end{align*}
with $X_t \in \R^n$, $Y_t, \bar{Y}_t \in \R^m$, a Lipschitz-continuous
function $V(x)$, constants
$\kappa, \varepsilon, \beta, \bar{\beta}, \bar{\gamma} > 0$ and
independent standard Brownian motions $B^X$, $B^Y$ on $\R^n$ and
$\R^m$.

Let $D \subset \R^m$ be a compact set and define the stopping time $\tau
= \inf\{ t \geq 0: Y_t \notin D \}$.

We will show that under some additional assumptions, a strong
averaging principle with rate $1/2$ holds in the sense that for any
fixed $T$ and $\varepsilon$ sufficiently small but fixed, there exists
a constant $C$ not depending on $\varepsilon$ such that
\begin{equation*}
  \sup_{0\leq t \leq T} \abs{Y_{t\wedge \tau} - \bar{Y}_{t \wedge \tau}} \leq C \varepsilon^{1/2}.
\end{equation*}

We need the following extra assumptions on the TAMD process:
\begin{gather*}
  0 < \lambda_\theta \Id_m < D\theta(x) {D\theta(x)}^T <
                                                         \Lambda_\theta
                                                         \Id_m
                                                         <\infty, \\
      -{(x_1-x_2)}^T (\nabla_x(\theta(x_1) - y)^2 -
      \nabla_x(\theta(x_2) - y)^2) \leq -\kappa_\theta
      \abs{x_1-x_2}^2 + \alpha_\theta \\
  \lim_{\abs{x}\to \infty} \abs{\theta(x)} = \infty \\
  \lambda_\theta \kappa > \Lambda_\theta \beta^{-1}.
\end{gather*}

In order to apply
Theorem~\ref{theorem:1} we also need to suppose that
Assumption~\ref{ass:rhoregularity} holds for the TAMD process.


We will now briefly comment on the form of $\bar{Y}_t$. Let
\begin{equation*}
  \mu(dx) = Z_0^{-1} e^{-V(x)} dx, \quad Z_0 = \int e^{-V(x)} dx
\end{equation*}
so that
\begin{align*}
  \bar{b}(y) &= \frac{Z_0}{Z(y)} \int -\bar{\gamma}^{-1} \kappa (\theta(x) - y)
  e^{-\tfrac\kappa2\abs{\theta(x) - y}^2} \mu(dx) \\
  &= \frac{Z_0}{Z(y)} \bar{\gamma}^{-1} \int -\kappa (z - y)
  e^{-\tfrac\kappa2\abs{z - y}^2} \theta_\#\mu (dz) \\
  &= \frac{Z_0}{Z(y)} \bar{\gamma}^{-1} \nabla_y \int
  e^{-\tfrac\kappa2\abs{z - y}^2} \theta_\#\mu (dz)
\end{align*}
where $\theta_\# \mu$ denotes the image measure of $\mu$ by $\theta$.
Now note that
\begin{equation*}
  \frac{Z(y)}{Z_0} = \int e^{-\tfrac\kappa2\abs{\theta(x) - y}^2}
                     \mu(dx) = \int
  e^{-\tfrac\kappa2\abs{z - y}^2} \theta_\#\mu (dz)
\end{equation*}
so that
\begin{equation*}
  \bar{b}(y) = \bar{\gamma}^{-1} \nabla_y \log \int
  e^{-\tfrac\kappa2\abs{z - y}^2} \theta_\#\mu (dz) = \nabla_y \log
  (\theta_\#\mu * \mathcal N(0, \kappa^{-1}))(y).
\end{equation*}
In the last expression, $*$ denotes convolution,
$\mathcal N(0, \kappa^{-1})$ denotes the Gaussian measure with mean
$0$ and variance $\kappa^{-1}$ and we identify through an abuse of
notation measures and their densities which we suppose to exist.

Thus,
\begin{equation*}
  d\bar{Y}_t = \bar{\gamma}^{-1} \nabla_y \log
  (\theta_\#\mu * \mathcal N(0, \kappa^{-1}))(\bar{Y}_t) dt + \sqrt{2
  {(\bar{\beta} \bar{\gamma})}^{-1}} dB^Y_t.
\end{equation*}
In physical terms, $\bar{Y}_t$ evolves
at an inverse temperature of $\bar{\beta}$
on the energy landscape corresponding to the image measure of $\mu$ by
$\theta$ convolved with a Gaussian measure of variance $\kappa^{-1}$.

We proceed to establish a Logarithmic Sobolev inequality for $\mu^y$
via the Lyapunov function method. From
\autocite{cattiaux_hitting_2017} Theorem 1.2 it follows that a
sufficient condition for a Logarithmic Sobolev inequality to hold for
an elliptic, reversible diffusion process with generator $L$ and
reversible measure $\mu$ is:
there exist constants $\lambda >
0$, $b > 0$, a function $W \geq w > 0$, a function $V(x)$ such
that $V$ goes to infinity at infinity, $\abs{\nabla V(x)} \geq v > 0$
for $\abs{x}$ large enough and such that $\mu(e^{aV}) < \infty$
verifying
\begin{equation*}
  LW(x) \leq -\lambda V(x) W(x) + b.
\end{equation*}

Fix $y$ and let $F(x, y) = \tfrac 12 \abs{\theta(x) - y}^2$. In order
to establish a Logarithmic Sobolev inequality for $\mu^y$ we are
going to show that the preceding condition holds for $V(x) = F(x, y)$
and $W(x) = e^{F(x, y)}$.  We have
\begin{align*}
  \nabla_x F(x) &= D\theta(x)^T (\theta(x)-y), \\
  \lambda_\theta \abs{\theta(x) - y}^2 & \leq \abs{\nabla_x F}^2 \leq
                             \Lambda_\theta \abs{\theta(x) - y}^2, \\
  \Delta F &= n \bar{\lambda}_\theta + (\Delta\theta)^T (\theta-y).
\end{align*}
Furthermore
\begin{align*}
 \varepsilon L^X F &= -\nabla_x V_0^T \nabla_x F 
  - \kappa \abs{\nabla_x F}^2 + \beta^{-1} \Delta F \\
  &= -\nabla_x V_0^T D\theta^T (\theta-y) - \kappa \abs{D\theta^T (\theta-y)}^2 + \beta^{-1} n
    \bar{\lambda}_\theta + \beta^{-1} \Delta\theta^T (\theta-y) \\
  & \leq \abs{\nabla_x V_0} \sqrt{\Lambda_\theta} \abs{\theta
    - y} -
    \kappa \lambda_\theta \abs{\theta - y}^2 + \beta^{-1} n
    \bar{\lambda}_\theta + \beta^{-1} \abs{\Delta \theta}
    \abs{\theta - y} \\
  & \leq -\kappa \lambda_\theta F + \frac{(\abs{\nabla_x V_0}
    \sqrt{\Lambda_\theta} +  \beta^{-1} \abs{\Delta
    \theta})^2}{2\kappa \lambda_\theta} +  \beta^{-1} n
    \bar{\lambda}_\theta \\
  &= -\kappa \lambda_\theta F + G(x)
\end{align*}
where we used the fact that $-ax^2 + bx + c \leq -\frac12 a x^2 +
\tfrac{b^2}{2a} + c$ for the second inequality.

Let $W(x, y) = e^{F(x, y)}$. Now,
\begin{align*}
  \varepsilon L^X W(x, y) &= \varepsilon L^X F(x, y) W(x, y) + \beta^{-1} \abs{\nabla_x F(x,
                y)}^2 W(x, y)\\
  & \leq -(\lambda_\theta \kappa - \Lambda_\theta \beta^{-1}) F(x, y) W(x, y) +
    \norm{G}_\infty W(x, y) \\
  &= -((\lambda_\theta \kappa - \Lambda_\theta \beta^{-1}) F(x, y) - \norm{G}_\infty)
    W(x, y).
\end{align*}
Since $F$ goes to infinity at infinity, for $x$ outside a compact set
\begin{equation*}
  -(\lambda_\theta \kappa - \Lambda_\theta \beta^{-1}) F(x, y) + \norm{G}_\infty \leq
  -\tfrac12 (\lambda_\theta \kappa - \Lambda_\theta \beta^{-1}) F(x, y)
\end{equation*}
so that
\begin{equation*}
  \varepsilon L^X W(x, y) \leq -\tfrac12 (\lambda_\theta \kappa -
  \Lambda_\theta \beta^{-1}) F(x,
  y) W(x, y) + K
\end{equation*}
for some constant $K$. This establishes a Log-Sobolev inequality for
the measure $\mu^y$ with respect to $\varepsilon \Gamma^X$ in the
sense that
\begin{equation*}
  \int f^2 \log f^2 d\mu^y \leq 2 c_L^y \int \varepsilon \Gamma^X d\mu^y
\end{equation*}
for some constant $c_L^y$ depending on $y$.
Let $c_L = \sup_{y\in D} c_L^y$ so that
\begin{equation*}
  \int f^2 \log f^2 d\mu^y \leq 2 \varepsilon c_L \int \Gamma^X d\mu^y.
\end{equation*}
This shows that a Log-Sobolev inequality with a constant
$\varepsilon c_L$ holds for each measure $\mu^y, y\in D$.

It remains to estimate $\kappa_X, \kappa_Y, \lipnorm{\partial_{y_i} U}^2, \lipnorm{\partial^2_{y_i} U}^2$ and $\bar{b}(y)^2$.

We have $b_X = -\varepsilon^{-1} \nabla_x V(x) - \varepsilon^{-1}
\tfrac\kappa2\nabla_x \abs{\theta(x) - y}^2$ and we want to find
$\kappa_X$ such that
\begin{equation*}
  {(x_1-x_2)}^T (b_X(x_1, y) - b_X(x_2, y)) \leq -\kappa_X
  \abs{x_1-x_2}^2 + \alpha \text{ for all } x_1, x_2 \in
  \R^n, y \in \R^m.
\end{equation*}

Since $\abs{\nabla_x V}$ is bounded and using the assumption on $\theta$, we
get
\begin{align*}
  \lefteqn{{(x_1-x_2)}^T (b_X(x_1, y) - b_X(x_2, y)) }\quad \\
  &= -\varepsilon^{-1} {(x_1-x_2)}^T (\nabla_x V(x_1) - \nabla_x V(x_2)) - \varepsilon^{-1}
\tfrac\kappa2 {(x_1-x_2)}^T (\nabla_x \abs{\theta(x_1) - y}^2 -
    \nabla_x \abs{\theta(x_2) - y}^2) \\
  & \leq - \varepsilon^{-1}
\tfrac\kappa2 \kappa_\theta \abs{x_1 - x_2}^2 + 2\varepsilon^{-1} \abs{x_1 - x_2} \norm{\nabla_x
    V(x)}_\infty + \varepsilon^{-1} \alpha_\theta \\
  & \leq -\varepsilon^{-1}\frac{\kappa \kappa_\theta}{4} \abs{x_1 -
    x_2}^2
    + 4 \varepsilon^{-1} \frac{\norm{\nabla_x V}_\infty}{\kappa
    \kappa_\theta} + \varepsilon^{-1} \alpha_\theta
\end{align*}
so that we can identify
\begin{align*}
  \kappa_X &= \varepsilon^{-1}\frac{\kappa \kappa_\theta}{4} 
& \alpha &= 4 \varepsilon^{-1} \frac{\norm{\nabla_x V}_\infty}{\kappa
    \kappa_\theta} + \varepsilon^{-1} \alpha_\theta.
\end{align*}

We have
\begin{equation*}
  b^i_Y(x, y) = -\nabla_{y_i} U(x, y) = -\kappa (y_i - \theta_i(x))
\end{equation*}
so that
\begin{equation*}
  \nabla_x b^i_Y(x, y) = \kappa \nabla_x \theta_i(x)
\end{equation*}
and
\begin{equation*}
  {\kappa_Y}^2 = \frac{1}{m} \sum_{i=1}^m {\kappa}^2 \norm{\nabla_x
    \theta_i(x)}_\infty^2 \leq {\kappa}^2 \Lambda_\theta.
\end{equation*}
We also have
\begin{equation*}
  \lipnorm{\partial_{y_i} U}^2 = \lipnorm{b_Y^i}^2 \leq \kappa^2 \norm{\nabla_x
    \theta_i}_\infty^2 \leq {\kappa}^2 \Lambda_\theta
\end{equation*}
and
\begin{equation*}
  \lipnorm{\partial^2_{y_i} U}^2 =   \lipnorm{\partial_{y_i} \kappa
    \theta(x)}^2 = 0
\end{equation*}
so that
\begin{equation*}
   {c_V}^2 = \sup_y 
  \left(\sum_{i=1}^m
  \lipnorm{\partial_{y_i} U(\cdot, y)}^2 + \sum_{i,j=1}^m
  \lipnorm{\partial^2_{y_i y_j} U(\cdot, y)}^2\right) \leq m
{\kappa}^2 \Lambda_\theta.
\end{equation*}

From the expression for $\varepsilon L^X F$ we get that
\begin{align*}
  F & \leq -\frac{\varepsilon}{\kappa \lambda_\theta}L^X F + \frac{G(x)}{\kappa \lambda_\theta}.
\end{align*}

Now
\begin{align*}
  \bar{b}^2(y) 
  &= \left(\int -\kappa(y - \theta(x)) \mu^y(dx) \right)^2 \\
  &\leq {\kappa}^2 \int F(x, y) \mu^y(dx) \\
  &\leq  -\frac{\kappa \varepsilon}{\lambda_\theta} \int L^X F(x, y) \mu^y(dx) +
    \frac{\kappa}{\lambda_\theta} \int G(x) \mu^y(dx) \\
  &=\frac{\kappa}{\lambda_\theta} \int G(x) \mu^y(dx) \\ 
\end{align*}
since $\mu^y$ is invariant for $L^X(\cdot, y)$.

The separation of timescales is
\begin{align*}
  \gamma = \frac{{\kappa_X}^2\lambda_Y}{\Lambda_x {\kappa_Y}^2} \geq
  \varepsilon^{-1} \frac{{\kappa_\theta}^2
  (\bar{\beta}\bar{\gamma})^{-1}}{16 \Lambda_\theta (\beta^{-1})}.
\end{align*}

If $\gamma > \frac{1}{(\sqrt{3} - \sqrt{2})^2}$
we can now apply Theorem~\ref{theorem:1} as in the previous section to
show that an averaging principle holds for the stopped TAMD process with rate
$\varepsilon^{1/2}$, i.e.
there exists a constant $C$ depending on $T, V, \beta_X,
b_Y$ and $\beta_Y$ such that for
\begin{equation*}
  \varepsilon \leq \frac{16 (\sqrt{3}-\sqrt{2})^2 \Lambda_\theta \bar{\gamma}
    \beta^{-1}}{
\kappa_\theta^2 \bar{b}^1}
\end{equation*}
we have
\begin{equation*}
  \E{\sup_{0\leq t \leq T} \abs{Y_{t\wedge\tau} - \bar{Y}_{t\wedge\tau}}} \leq \sqrt{\varepsilon} C.
\end{equation*}



\section*{Acknowledgement}
The author wants to thank Tony Lelièvre for 
suggesting the entropy approach in section~\ref{sec:dist}.

\printbibliography

\end{document}